\documentclass{article}
\usepackage{amsmath}
\usepackage{amsfonts}
\usepackage{hyperref}
\usepackage{amssymb}
\usepackage{graphicx}
\usepackage{amscd}
\providecommand{\U}[1]{\protect\rule{.1in}{.1in}}
\newtheorem{theorem}{Theorem}

\newtheorem{conjecture}[theorem]{Conjecture}

\newtheorem{definition}[theorem]{Definition}

\newtheorem{lemma}[theorem]{Lemma}

\newtheorem{problem}[theorem]{Problem}
\newtheorem{proposition}[theorem]{Proposition}
\newtheorem{remark}[theorem]{Remark}

\newenvironment{proof}[1][Proof]{\noindent\textbf{#1.} }{\ \rule{0.5em}{0.5em}}
\begin{document}


\title{Noncommutative Topology \\ and \\Prospects in Index Theory. \\
\author \large{Nicolae Teleman} \\ 
Dipartimento di Scienze Matematiche, \\
Universita' Politecnica delle Marche, 60131-Ancona, Italia   \\
 e-mail:   teleman@dipmat.univpm.it }
\maketitle

\thanks{\large 
\noindent
\emph{This article is a tribute to the memory of \\
 Professor Enzo Martinelli, \\ with deep respect and reconesance.\\
   Nicolae Teleman}


\section{Abstract} 
{
The index formula is a \emph{local} statement made on global and local data; for this reason we introduce \emph{local Alexander - Spanier co-homology}, \emph{local periodic cyclic homology},    \emph{local Chern character} 
and  \emph{local} $T_{\ast}$-theory.
Index theory should be done: {\bf Case 1}:  for \emph{arbitrary rings},  {\bf Case 2}: for  \emph{rings of functions over topological manifolds}.
{\bf Case 1}   produces general index theorems, as for example, over \emph{pseudo-manifolds}. 
{\bf Case 2}  gives a general treatment of classical and non-commutative index theorems.   
All existing index theorems belong to the second category.
The \emph{tools of the theory} would contain:  \emph{local} $T_{\ast}$-theory,  \emph{local periodic cyclic  homology},  \emph{local Chern character}.  These tools are extended to \emph{non-commutative topology}.
The index formula has \emph{three stages}: 
{\bf Stage I} is done in $T^{loc}_{i}$-theory,   {\bf Stage II}  is done in the local periodic cyclic homology and  {\bf Stage III}  involves products of distributions, or restriction to the diagonal.  For each stage there corresponds a {\bf topological index}  and an {\bf analytical index}. 
The construction of $T_{\ast}$-theory involves the $T$-{\bf completion}. It involves also the need to work with {\bf half integers};  this should have important consequences.


La formule de l'indice est une \emph formule {locale} faite sur les donn�es globales et locales; pour cette raison, nous introduisons \emph {local Alexander  - Spanier co-homologie}, \emph {homologie cyclique p�riodique locale}, \emph {caract�re de Chern locale}
et la  th�orie \emph {local} $T_ {\ast} $.
La th�orie de l'indice doit �tre fait: {\bf Case 1}: pour \emph{anneaux arbitraires}, {\bf Cas 2}: pour \emph {anneaux de fonctions sur les varietes topologiques}.
{\bf Cas 1} est le cas des th�or�mes de l'indice g�n�ral, comme par exemple, par example \emph {pseudo-varietees}.
{\bf Cas 2} donne un traitement g�n�ral des th�or�mes d'index classiques et de la geometrie non-commutatives.
Tous les th�or�mes d'index existants appartiennent � la deuxi�me cat�gorie.
Le \emph {outils de la th�orie} contiendra: \emph {local} $T _ {\ast}$ - th�orie, \emph {homologie cyclique p�riodique locale}, \emph {caract�re de Chern locale}. Ces outils sont �tendus � \emph {topologie non-commutative}.
La formule de l'indice a \emph {ha trois} �tapes:
{\bf Etape I} est fait dans $T ^ {loc} _ {i}$ - th�orie, {\bf Etape II} est fait dans l'homologie cyclique p�riodique local et {\bf Etape III} implique des produits de distributions, ou la restriction � la diagonale. Pour chaque �tape correspond un {\bf index topologique} et une {\bf index analytique}.
La construction de $T _ {\ast}$ - th�orie utilises le $T$-{\bf completion}. Elle implique �galement la n�cessit� de travailler avec {\bf demi entiers}; cela devrait avoir des cons�quences importantes.
}

\part{\emph{Local} algebraic structures.} 
 
\section{Localised rings.}   
\begin{definition}      \label{localisedrings}   
\emph{Localised rings}.
\par
Let $\mathcal{A}$ be an unital associative ring.
The ring $\mathcal{A}$ is called \emph{localised ring} provided it is endowed with an additional structure satisfying the axioms (1) - (3) below.
\par  
Axiom 1. The underlying space $\mathcal{A}$ has a decreasing filtration by sub-spaces
$\{ \mathcal{A}_{\mu}   \}_{n \in \mathbb{N}} \subset \mathcal{A}$.
\par  
Axiom 2.
$\mathbb{C}. 1 \subset \mathcal{A}_{\mu}, \; for\;  any \; \; \mu \in \mathbb{N}$
\par 
Axiom 3.
 For any $\mu, \; \mu' \in \mathbb{N}^{+}$,   \;  $ \mathcal{A}_{\mu} \; . \; \mathcal{A}_{\mu'}   \subset \mathcal{A}_{\mathit{Min}(\mu,\mu')-1}$,  ($ \mathcal{A}_{0} \; . \; \mathcal{A}_{0} \subset   \mathcal{A}_{0}$).
\end{definition}  

\begin{remark}  
 A ring $\mathcal{A}$ could have different localisations.
\end{remark}  

\section{\emph{Local} Alexander - Spanier co-homology.}

\begin{definition}
Let $\mathcal{A} = \mathcal{A}_{\mu}$ be a localised unital ring and $G$  an Abelian  group. Define
\begin{equation}
C_{AS}^{p}  (\mathcal{A}_{\mu}, \; G) \;=\; \{ \;  \sum_{i}\;   g_{i} a_{0}^{i} \otimes    a_{1}^{i} \otimes \dots \otimes a_{p}^{i}, \;\;\;
g_{i} \in G, \;\; a_{i} \in   \mathcal{A}_{\mu} \;
         \}_{p= 0, 1, \dots \infty}.
\end{equation}
The boundary map $d$  is defined as in the classical definition of Alexander - Spanier co-homology
\begin{gather}
d:  C_{AS}^{p}  (\mathcal{A}_{\mu}, \; G) \longrightarrow C_{AS}^{p+1}  (\mathcal{A}_{\mu}, \; G) \\
d  ( g \;a_{0} \otimes a_{1}^{i} \otimes \dots \otimes a_{p}) \;= \\
g \;[  1\otimes  a_{0} \otimes a_{1}^{i} \otimes \dots \otimes a_{p} -
 a_{0} \otimes 1 \otimes a_{1}^{i} \otimes \dots \otimes a_{p}  + \\ \dots  +   (-1)^{p+1} a_{0} \otimes a_{1}^{i} \otimes \dots \otimes a_{p} \otimes 1 ].                              
\end{gather}
The \emph{local} Alexander - Spanier co-homology is
\begin{equation}
H_{AS}^{loc, p} (\mathcal{A}) \;=\; ProjLim_{\mu}  H_{p}   (C_{AS}^{\ast}  (\mathcal{A}_{\mu}, \; G), \; d). 
\end{equation}
\end{definition}
\section{\emph{Local}  periodic cyclic homology. Long exact sequence. }  \label{LocalHomologyTheory}
We assume  that $\mathcal{A}_{\mu}$ is a localised unital ring.
\par
The operators, see \cite{Connes_IHES},  $T$, $N$, $B$, $I$, $S$   pass to localised rings.
Therefore \emph{local cyclic homology}, \emph{local periodic cyclic homology} may be extended to localised rings.  
\begin{theorem} \label{longexactsequence20.2}   
The \emph{local Hochschild and cyclic homology} are well defined for a localised ring.
\par
One has the exact sequence (analogue of Connes' exact sequence, see Connes \cite{Connes_IHES})
\begin{equation}
\dots
\overset{B}{\longrightarrow}  
    H^{loc}_{p}  (\mathcal{A})
\overset{I}{\longrightarrow}  
     H^{loc, \lambda}_{p}  (\mathcal{A})
\overset{S}{\longrightarrow}  
    H^{loc}_{p-2}  (\mathcal{A})
\overset{B}{\longrightarrow} 
    H^{loc}_{p-1}  (\mathcal{A})
\overset{S}{\longrightarrow}     
\dots
\end{equation}
\end{theorem}    
\begin{theorem}
For  $\mathcal{A} = C^{\infty}(M)$, one has
\begin{equation}
H^{loc}_{p} (C^{\infty (M))}  \;= \;  H^{loc}_{p} (C^{\infty (M))}. 
\end{equation}
\end{theorem}
\par
The \emph{local bi-complex} $(b, B) (\mathcal{A}_{\mu})$ is well defined for localised rings too.
The general term of the $(b, B) (\mathcal{A}_{\mu})$ bi-complex is
\begin{gather}
C_{p, q} \; =\; \otimes^{p-q+1} \mathcal{A}_{\mu},   \;\;\;  q \leq p \\
b:  C_{p, q}    \longrightarrow   C_{p, q-1}    \\
B:  C_{p, q}    \longrightarrow   C_{p+1, q}. 
\end{gather}

\begin{definition}
The homology of the bi-complex $(b, B) (\mathcal{A}_{\mu})$  has two components:  the \emph{even}, resp. \emph{odd},  component corresponding to the $p - q = even\;number$, resp.  $p - q = odd\;number$.
\par
The \emph{local periodic cyclic homology}  of the localised ring $\mathcal{A}$ is
\begin{equation}
H^{loc, per, \lambda}_{ev, odd} (\mathcal{A})\; := \; LimProj_{\mu}  \; H_{ev, odd} H_{\ast}(\mathcal{A}_{\mu}).
\end{equation}
\end{definition}  
 
\section{\emph{Local} periodic Chern character}

In this section we localise  the periodic even/odd Chern character.
All operations involved in the construction  of the cyclic homology
may be localised. Here we use the bi-complrx
$(b, B)(\mathcal{A}_{\mu})$ only onto \emph{non degenerate} elements, i.e. onto the range of the idempotent $\Pi$.
\begin{definition} \emph{local periodic cyclic Chern character}.
\par
Let  $\mathcal{A}_{\mu}$ be a localised ring. 
\begin{enumerate}
\item
For any idempotent $e \in \mathbb{M}_{n}(\mathcal{A}_{\mu})$ define
\begin{equation}
Ch_{ev, \mathcal{A}_{\mu}}\; (e)\;=\;  Tr \;e \; + \sum_{p=1}^{\infty}  \; (-1)^{p} \frac{(2p) !}{p !}  \; ( e - \frac{1}{2}) (de)^{2p}.
\end{equation}
$Ch_{ev, \mathcal{A}_{\mu}} \; (e)$ is an even cycle in the $(b, B)(\mathcal{A}_{\mu})$ cyclic complex.  Its homology class is the
\emph{periodic cyclic Connes Chern character of} $e$.
\par
Let  $e \in T^{loc}_{0}(\mathcal{A})$;  define
\begin{equation}
Ch_{ev} \; (e)  \; = \; ProjLim_{\mu} \; Ch_{ev, \mathcal{A}_{\mu}}\; (e)  \in H^{loc, per, \lambda}_{ev}(\mathcal{A}).
\end{equation} 
\item
For any invertible $u \in \mathbb{M}_{n}(\mathcal{A}_{\mu})$,  define
\begin{equation}
Ch_{odd}\; (u)\;=\;  \sum_{p} (-1)^{2p+1} (2p+1)! \; Tr \; (u^{-1} du) . (du^{-1} du). \dots (du^{-1} du).
\end{equation}
\par
Let  $u \in T^{loc}_{1}(\mathcal{A})$;  define
\begin{equation}
Ch_{odd} \; (u)  \;=\;  ProjLim_{\mu} \; Ch_{odd, \mathcal{A}_{\mu}}\; (u) \in H^{loc, per, \lambda}_{odd}(\mathcal{A}).
\end{equation} 
\end{enumerate}
\end{definition}
In non-commutative topology, discussed later in this article, we will present a definition of the Chern character of idempotents.

\section{\emph{Local index theorem}}  
In this section we consider an elliptic pseudo-differential operator $A$.  We are going to define the \emph{Chern character} of differences of idempotents; the skew-symmetrisation is not necessary. For more details about this topic see
Teleman  \cite{Teleman_arXiv_I},  \cite{Teleman_arXiv_IV}.
\begin{lemma}  
Let $R(A) :=  P - e$, where $P$  and $e$ are idempotents.
Then $R(A)$  satisfies the identity
\begin{equation}  
R(A)^{2} = R(A) - (e . R(A)  + R(A). e).
\end{equation}  
\end{lemma}  
\begin{definition}  
Let $R(A) :=  P - e$, where $P$  and $e$ are idempotents. Let  $\Lambda = \mathbb{C} + e . \mathbb{C}$.
Then, 
\begin{equation}
C_{q} \;(R(A)) :=  Tr  \; \otimes_{\Lambda}^{2q+1}  \; R(A)
\end{equation}
is the \emph{Chern character} of $R(A)$. For the definition of $R(A)$ see  \S 23, Definition 57.  
\end{definition}   
\begin{lemma}  
$C_{q} \;(R(A))$  is a cycle in the complex  
$\{   \otimes_{\Lambda}^{\ast} \mathbb{M}_{N}(\mathcal{A})  \}_{\ast}$, such that
\begin{equation}  
\lambda.a_{0} \otimes_{\Lambda} a_{1} \otimes_{\Lambda} \dots  \otimes_{\Lambda} a_{k} = 
a_{0} \otimes_{\Lambda} a_{1} \otimes_{\Lambda} \dots  \otimes_{\Lambda} a_{k} .\lambda
\end{equation}  
for any $\lambda \in \Lambda$.
\end{lemma}  
\begin{proof}
One has
\begin{gather}  
b^{'} \; C_{q} \;(R(A)) \;=\;  
Tr \;[ 
\;(R(A) - (e . R(A)  + R(A). e)) \otimes_{\Lambda} \dots \otimes_{\Lambda} R(A) \dots 
 \\
- R(A)  \otimes_{\Lambda} R(A) \otimes_{\Lambda} \dots \otimes_{\Lambda} (R(A) - ( e . R(A) + R(A).e )) \; 
] \;= \\
Tr \; ;[ \; - (e . R(A)  + R(A). e)) \otimes_{\Lambda} \dots \otimes_{\Lambda} R(A) \dots  \\
- R(A)  \otimes_{\Lambda} R(A) \otimes_{\Lambda} \dots \otimes_{\Lambda} ( - ( e . R(A) + R(A).e ) \; ] = 
\\
- Tr \; [ \; e . R(A) \otimes_{\Lambda}  \dots \otimes_{\Lambda} R(A) \;-\;  R(A) \otimes_{\Lambda}  \dots \otimes_{\Lambda} R(A) .e \;] \;=\; 0.
\end{gather}   
$C_{q} \;(R(A))$ is a cyclic cycle. It represents a cyclic homology class in 
$H_{\ast}^{\lambda} (\otimes^{2q+1} ( \mathcal{A}_{\mu})$. \par 
The ring  $\Lambda$ is separable. Indeed the mapping
\begin{equation*}
\mu: \Lambda \otimes \Lambda^{op}  \longrightarrow \Lambda,  \;\;\; 
\mu (1 \otimes 1) = 1, \;  \mu(1 \otimes e) =  \mu(e \otimes 1) = e, \;\;  \mu (e \otimes e) = e
\end{equation*} 
has the $\Lambda$-bimodule splitting 
\begin{equation*}
s:  \Lambda  \longrightarrow  \Lambda \otimes \Lambda^{op}, \;\;\; 
s(1) = e \otimes e +  (1-e) \otimes   (1-e),   \;\;  s(e) = e \otimes e.
\end{equation*}
Theorem  1.2.13 \cite{Loday} states that the homology of the Hochschild complex is isomorphic to the
homology of the complex  $C^{S}( \mathcal{A})$.  
Using the long exact sequence in homology associated to the complex $H_{\ast}^{\lambda} (\otimes^{2q+1} ( \mathcal{A}_{\mu})$ and the Connes exact sequence,  we find that the cyclic homology of  $H_{\ast}^{\lambda} (\otimes^{2q+1} ( \mathcal{A}_{\mu})$
is isomorphic to the cyclic homology of the algebra $\mathcal{A}$.  Here we have used the localisation
of the algebra $\mathbb{M}_{N}( \mathcal{A})$  defined by the supports of the operators.
\end{proof}
It remains to solve the problem of producing the \emph{whole index class}. This could be obtained by the formula
\begin{equation}
Ch_{ev} \;R(A) -  Ch_{ev} e_{II},         
\end{equation}
where $R(A)$ is the the operator associated to the signature operator.  


\part{Prospects in Index Theory}  

\section{\emph{Local} algebraic $T_{i}$ theory.}  


This section deals with the program presented in  \cite{Teleman_arXiv_III}.
\par
For any associative ring $\mathcal{A}$ we define the \emph{commutative groups} 
$T_{i}(\mathcal{A})$, $i = 0, 1$.  We introduce the notion of \emph{localised rings}, see Definition \ref{localisedrings},  
$\mathcal{A} = \{ \mathcal{A}_{\mu} \}$, given by a \emph{linear filtration} of the algebra $\mathcal{A}$ and  we associate the \emph{commutative groups} $T_{i}^{loc} (\mathcal{A})$.  
Although we define solely $T^{loc}_{i}(\mathcal{A})$ for $i = 0, 1\; $, we expect our construction could be extended in higher degrees.  
\par
\emph{We stress that our construction of} $T_{i}(\mathcal{A})$ and
$T^{loc}_{i}(\mathcal{A}), $\;i = 0, 1$, $ \emph{uses exclusively matrices}.   The projective modules are totally avoided. 
The role of \emph{equivalence of projective modules}, used in the classical construction of the algebraic $K$-theory, is played by \emph{conjugation}.  
\par
\emph{The commutative group}  $T_{0}(\mathcal{A})$ is by definition the \emph{Grothendieck completion of the space of idempotent matrices factorised through the equivalence relations}: -i) \emph{stabilisation} $\sim_{s}$, -2) \emph{conjugation} $\sim_{c}$,   and  -3) for \emph{localised groups},   $T_{0}^{loc}(\mathcal{A})$,  \emph{projective limit with respect to the filtration}, denoted $\sim_{p}$. The \emph{non-localised} group $T_{0}(\mathcal{A})$ coincides with the classical algebraic $K_{0}$-theory group.
\par
The groups $T_{1}(\mathcal{A})$ are by definition the quotient of $\mathbb{GL}(\mathcal{A})$ through
the equivalence relation generated by
  -1) stabilisation $\sim_{s}$,  -2) \emph{conjugation} $\sim_{c}$ and -3) 
  $\sim_{\mathbb{O}(\mathcal{A})}$, where $\mathbf{O}(\mathcal{A})$  is the \emph{sub-group generated by elements of the form} $u \oplus u^{-1} $,  \emph{for any} $u \in \mathbb{GL}(\mathcal{A})$.   
For any    $u_{1},  u_{2} \in \mathbf{GL}(\mathcal{A}) \;/\;  (\sim_{s} \cup \sim_{c} ) $,  one defines
$u_{1} \sim_{\mathbf{O}(\mathcal{A})}  u_{2}$   provided there exist   $\xi_{1}, \xi_{2} \in  \mathbb{O}(\mathcal{A})$  such that
$u_{1} + \xi_{1} = u_{2} + \xi_{2} $. The third relation is a particular case of a new completion procedure which we call \emph{T-completion}.
The operation   
\begin{equation*}
\mathbb{GL}(\mathcal{A}) \;/\; ( \sim_{s} \cup \sim_{c} )  \longrightarrow
\mathbb{GL}(\mathcal{A}) \;/\; ( \sim_{s} \cup \sim_{c} \cup \sim_{\mathbf{O}(\mathcal{A})}  )
\end{equation*}
transforms the commutative \emph{semi-group} $\mathbb{GL}(\mathcal{A}) \;/\; ( \sim_{s} \cup \sim_{c} )$
in the commutative \emph{group}  $\mathbb{GL}(\mathcal{A}) \;/\; ( \sim_{s} \cup \sim_{c} \cup \sim_{\mathbf{O}(\mathcal{A})}  )$.  
\par
The groups $T^{loc}_{i}(\mathcal{A})$ follow the same construction as that of $T_{i}(\mathcal{A})$,
provided the supports of the elements belong to $\mathcal{A}_{\mu}$.
\emph{Our definition of}  $T_{1}(\mathcal{A})$ and $T^{loc}_{1}(\mathcal{A})$ \emph{does not use the commutator sub-group} $[\mathbf{GL}(\mathcal{A}), \mathbf{GL}(\mathcal{A})]$  \emph{nor elementary matrices in its construction}.  
\par
We define \emph{short exact sequences of  localised rings}  and we get the corresponding open six terms exact sequence (Theorem \label{exactsequence} ).
\par
We stress that one  has to take the tensor product of the expected six terms exact sequence by 
$\mathbb{Z}[\frac{1}{2}]$ in order to get the open six terms exact sequence.   We expect the factor 
$\mathbb{Z}[\frac{1}{2}]$ to have important implications, among them,  Pontrjagin classes, existence of a generator of the $K$-homology fundamental class and Kirby-Siebenmann obstruction class.
\par
Our work shows that the basic relations  which define $T_{1}$ and $T_{1}^{loc}$  reside in the \emph{additive} sub-group generated by elements of the form $u \oplus u^{-1}$,   $u \in \mathbf{GL}(\mathcal{A})$,  rather than in the \emph{multiplicative} commutator sub-group $[\mathbf{GL}(\mathcal{A}), \mathbf{GL}(\mathcal{A})]$.   
\par
Even for trivialy filtered algebras,  $\mathcal{A} = \{ \mathcal{A}_{\mu} \}$,  for all $\mu \in \mathbb{N}$, the groups $T^{loc}_{1}(\mathcal{A})$ provides more information than the classical group $T_{1}(\mathcal{A})$.
For the computation of the groups  $T_{i}^{loc}(\mathbb{C})$ see 
\cite{Teleman_arXiv_V}.
\par
\section{Motivation}   \label{21Introduction}  
To motivate the next considerations we have in mind the index formula applied onto the algebra of integral operators and pseudo-differential operators. The index formula is a \emph{global statement} whose ingredients  \emph{may be computed by local data}.  Our \emph{leading idea} is to \emph{localise} $K$-theory and periodic cyclic homology along the lines of the Alexander-Spanier co-homology in such a way that the new tools operate naturally with it,  see \cite{Teleman_arXiv_III}.
\par
This article goes along the lines of  \cite{Teleman_arXiv_III}. We define \emph{localised rings} $\mathcal{A}$ and we define and their \emph{local} $T$-theory, $T^{loc}_{i}(\mathcal{A})$, $i = 0, 1$. 
Keeping in mind that both the topological and analytic indices of an elliptic operator are stable under cutting
the operators about the diagonal, leads us naturally to the notion of \emph{localised rings}.  
Based on this notion we pass to the problem of finding a natural $K$-theory able to controle these \emph{local entities}. These are the \emph{local algebraic} $T$-theory groups, $T^{loc}_{i}(\mathcal{A})$, $i=0,1$ and \emph{local cyclic homology}.   
\par 
For the construction of  $K_{i} $-theory groups, in the pure algebraic context, the reader may consult the following books \cite{Milnor_Introduction}, \cite{Karoubi_1987},  \cite{Rosenberg},  \cite{Weibel},  in the Banach algebras or $C^{\ast}$-algebras category, \cite{Rordam_Lansen_Lautsen},  \cite{Blackadar}. 
\par
Our construction of $T^{loc}_{0}$ uses exclusively idempotent matrices.  The reasons why we chose to avoid projective modules are: -i) projective modules, in comparison with idempotent matrices,  contain more ambiguity and -ii) matrices are more suitable for controlling the ring filtration data and are more prone to make calculations.
No reference to projective modules is used in our constructions. 
\par
 Regarding our construction of $T^{loc}_{1}$ we recall that the \emph{classical algebraic} $K$-theory group 
 $K_{1}(\mathcal{A})$ of the algebra $\mathcal{A}$, see   \cite{Whitehead}, \cite{Bass}, \cite{Milnor_Introduction}, \cite{Karoubi_1987},  is by definition the \emph{Whitehead group}  
\begin{equation}  
K_{1}(\mathcal{A}) := \mathbf{GL}(\mathcal{A})/ [\mathbf{GL}(\mathcal{A}), \mathbf{GL}(\mathcal{A})],
\end{equation}     
 where
$[\mathbb{GL}(\mathcal{A}), \mathbb{GL}(\mathcal{A})]$ is the commutator normal sub-group of the group of invertible matrices
$\mathbb{GL}(\mathcal{A})$.  Our definition of \emph{local} $T$-theory groups needs to keep track of the \emph{number of multiplications} performed inside the algebra $\mathcal{A}$. In order for our constructions to hold it is necessary to involve a \emph{bounded} number of multiplications. \emph{It is important to state that, in general, the number of multiplications needed to generate the whole commutator sub-group is un-bounded}. 
It is known that the commutator sub-group is also generated by the \emph{elementary} matrices.
This is the reason why our definition of $T^{loc}_{1}(\mathcal{A})$ avoids entirely factorising $\mathbb{GL}(\mathcal{A})$ through the commutator sub-group or the sub-group generated by elementary matrices. 
\par
$T_{0}^{loc}(\mathcal{A})$ is by definition the Grothendieck completion of the semi-group of idempotent matrices in 
$\mathbb{M}_{n}(\mathcal{A}_{\mu})$ modulo three equivalence relations:  -i) stabilisation $\sim_{s}$, -ii) local conjugation $\sim_{c}$ by invertible elements $u \in  \mathbb{GL}_{n}(\mathcal{A}_{\mu})$ and -iii) projective limits with respect to $\mu \in \mathbb{N}$.
\par
To understand the relationship between our definition of the group $T_{1}^{loc}$  and $K_{1}$, recall  that the commutator sub-group $[\mathbb{GL}(\mathcal{A}, \mathbf{GL}(\mathcal{A} ]$ is given by arbitrary products of multiplicative commutators
$$
[A, B] := ABA^{-1}B^{-1},  \quad for\;any \; A, B \in \mathbf{GL}_{n}(\mathcal{A}).
$$ 
On the other side, supposing that $A$ and $B$ are conjugated, i.e.  $A = U B U^{-1}$,   we have
$$
A = U B U^{-1} = U B U^{-1} . B^{-1} B = [ U, B ].B. 
$$ 
This shows that if $A$ and $B$ are conjugated, they differ, multiplicatively, by a commutator. To complete this remark, we say that $A$ and $B$ are \emph{locally conjugated}  provided 
$A, B$ and $U$ belong to some $\mathbb{GL}_{n}(\mathcal{A}_{\mu})$;
here, $\mathcal{A}_{\mu}$ denote the terms of the filtration of $\mathcal{A}$.
\par
It is important to note that in the particular case $\mathcal{A} = \mathbb{C}$ 
the quotient of $\mathbb{GL}(\mathbb{M}(\mathbb{C}))$
through the commutator sub-group   
$\mathbb{GL}(\mathbb{M}(\mathbb{C})), \mathbb{GL}(\mathbb{M}(\mathbb{C}))]$ 
gives much less information than $T_{1}^{loc}(\mathbb{M}(\mathbb{C}))$.
\subsection{Factorisation of $u \otimes u^{-1}$}   
\begin{proposition}   
In the algebra of matrices one has the identity 
\begin{gather}  \label{factorisation} 
\begin{pmatrix}
u &  0 \\
0 &  u^{-1}
\end{pmatrix}
= 
\begin{pmatrix}
1 &  u \\
0 &  1
\end{pmatrix}
\begin{pmatrix}
1 &  0 \\
-u^{-1} &  1
\end{pmatrix}
\begin{pmatrix}
1 &  u \\
0 &  1
\end{pmatrix}
\begin{pmatrix}
0 &  -1 \\
1 &  0
\end{pmatrix}
.
\end{gather}    
\end{proposition}  
\par
This formula will play an important role in the construction of $T_{\ast}$-theory.

\section{Algebraic $T_{i}$ and $T_{i}^{loc}$-theory.}   

\subsection{Generalities and Notation.}  
Let $\mathcal{A}$ be a ring, with or without unit. If the unit will be needed, the unit will be adjoined.
\par
\begin{definition}  
Given the ring $\mathcal{A}$ we denote by $\mathbb{M}_{n}(\mathcal{A})$ the space of $n \times n$ matrices with entries in $\mathcal{A}$.  $\mathbb{M}_{n}(\mathcal{A})$ is a bi-lateral $\mathcal{A}$-module.
\par
Let  $\mathbb{I}demp_{n} \subset \mathbb{M}_{n}(\mathcal{A})$ be the subset of idempotents $p$ ($p^{2} = p$) of size $n$ with entries in $\mathcal{A}$.
\par
Suppose $\mathcal{A}$ has an unit. We denote by $\mathbb{GL}_{n}(\mathcal{A})$ the sub-space of matrices $M$ of size $n$, with entries in $\mathcal{A}$ which are invertible,  i.e. there exists the matrix $M^{-1} \in \mathbb{M}_{n}(\mathcal{A})$ such that $M M^{-1} = M^{-1} M = 1$. $\mathbb{GL}_{n}(\mathcal{A})$ is a non-commutative group under multiplication.
\end{definition}     
\begin{definition}  
The inclusions 
\par
-i) $\mathbb{I}demp_{n}(\mathcal{A}) \longrightarrow \mathbb{I}demp_{n+1}(\mathcal{A}) $
\begin{equation}   
p \mapsto
p^{'} =
\begin{pmatrix}
p  & 0 \\
0  &  0
\end{pmatrix}
\end{equation}  
\par
-ii) $\mathbb{GL}_{n}(\mathcal{A}) \longrightarrow \mathbb{GL}_{n+1}(\mathcal{A}) $
\begin{equation}  
u \mapsto 
u^{'} =
\begin{pmatrix}
u  & 0 \\
0  &  1
\end{pmatrix}
\end{equation}  
are  \emph{stabilisations}.
\par
Stabilisations of idempotents -i) and invertible matrices -ii) define two direct systems with respect to $n \in \mathbb{N}$.
\par 
-iii) If $p$ or $p^{'}$ are idempotents and one of them is an iterated stabilisation of the other we write  $p \sim_{s} p^{'}$.
\par
 If $u$ or $u^{'}$ are invertible matrices and one of them is an iterated stabilisation of the other we write   $u \sim_{s} u^{'}$.
\end{definition}  


\section{\emph{Localised rings}}   
We recall the definition of localised rings given before.
\begin{definition}       
\par
Let $\mathcal{A}$ be an unital associative ring.
The ring $\mathcal{A}$ is called \emph{localised ring} provided it is endowed with an additional structure satisfying the axioms (1) - (4) below.
\par  
Axiom 1. The underlying space $\mathcal{A}$ has a decreasing filtration by sub-spaces
$\{ \mathcal{A}_{\mu}   \}_{n \in \mathbb{N}} \subset \mathcal{A}$.
\par  
Axiom 2.
$\mathbb{C}. 1 \subset \mathcal{A}_{\mu}, \; for\;  any \; \; \mu \in \mathbb{N}$
\par 
Axiom 3.
 For any $\mu, \; \mu' \in \mathbb{N}^{+}$,   \;  $ \mathcal{A}_{\mu} \; . \; \mathcal{A}_{\mu'}   \subset \mathcal{A}_{\mathit{Min}(\mu,\mu')-1}$,  ($ \mathcal{A}_{0} \; . \; \mathcal{A}_{0} \subset   \mathcal{A}_{0}$).
\end{definition}  
\begin{definition}    \emph{Homomorphisms of localised rings.   Induced homomorphism.} 
\par
-i) A homomorphism from the localised ring  $ \mathcal{A} = \{ \mathcal{A}_{\mu}   \}_{\mu \in \mathbb{N}}$ to the localised ring  $ \mathcal{B} = \{ \mathcal{B}_{\mu}   \}_{\mu \in \mathbb{N}}$ is an ring homomorphism  
$\phi:   \mathcal{A} \longrightarrow  \mathcal{B}$  such that
$ \phi: \mathcal{A}_{\mu}    \longrightarrow \mathcal{B}_{\mu}$, for any $\mu \in \mathbb{N}$. 
\par
-ii) Let $f:  \mathcal{A} \longrightarrow  \mathcal{B}$  be a localised ring homomorphism.
\par
Let $f_{\ast}:  \mathbb{M}_{n}(\mathcal{A}_{\mu}) \longrightarrow  \mathbb{M}_{n}(\mathcal{B}_{\mu}) $ be the induced homomorphism which replaces any component $a_{ij}$ of the matrix $M$ with the component $f(a_{ij})$ of the matrix  $f_{\ast} (M)$. 
\end{definition}    
\begin{remark}   
The notion of localised ring differs from the notion of \emph{m-algebras} defined by Cuntz \cite{Cuntz}, in many respects. The sub-spaces $\mathcal{A}_{n}$ are not required to be \emph{algebras}, or even more, \emph{topological algebras}. In the Cuntz' definition of localised Banach algebras, the projective limit of these sub-algebras might be the zero Banach algebra. However, even in these cases, the corresponding \emph{local} $T$-theory could be not trivial.
\end{remark}   
\begin{remark}  
In this section any localisation of the algebra of bounded operators on the Hilbert space 
$H := L_{2} (M \times M)$ derives from a \emph{decreasing filtration} 
$\{ \mathcal{U}_{\mu} \}$,  towards the diagonal, of the space of bounded operators on 
$ M \times M$.  In such a case, if $\mu^{'} > \mu$,  one has the close inclusions
\begin{equation}  
L_{2}( \mathcal{U}_{\mu^{'}})  \subset  L_{2} ( \mathcal{U}_{\mu}) 
\end{equation}   
and  \emph{any internal authomorphism}  of the algebra 
$\mathbb{B}(\mathcal{U}_{\mu^{'}})$
is an internal authomorphism of the algebra  $\mathbb{B}(\mathcal{U}_{\mu})$
 (any auto-morphism of the algebra $L_{2}( \mathcal{U}_{\mu})$  will be extended by the identity on the complement of the Hilbert space $L_{2}( \mathcal{U}_{\mu^{'}})$).
 \end{remark}   
\begin{remark}  
The immediate application of this theory regards pseudo-differential operators.
The pseudo-differential operators on a compact smooth manifold form a localised (Banach) algebra. The filtration  is defined in terms of the support of the operators; the bigger the filtration order is, the smaller the supports of the operators towards the diagonal are.
\end{remark}  


\section{Local Mayer-Vietoris Diagrams.}   \label{MayerVietoris}   
In this section we \emph{adapt} Milnor's \cite{Milnor_Introduction} desciption of the first two algebraic $K$-theory groups \emph{to the case of localised rings}.
\par
Let  $\Lambda, \Lambda_{1}, \Lambda_{2}$ and $\Lambda'$ be rings with unit $1$ and let
\begin{equation}   
\begin{CD}
\Lambda            @>i_{1}>>   \Lambda_{1}\\
@VV{i_{2}} V                 @VV{j_{1}}V\\
\Lambda_{2}     @>j_{2}>>    \Lambda'
\end{CD}
\end{equation}   
be a commutative diagram of ring homomrphisms. All ring homomorphisms $f$ are assumed to satisfy  
$f(1) = 1.$ 
Any module in this paper is a left module.
\par
We assume the diagram satisfies the three conditions below.
\par
$\mathit{Hypothesis}$ 1.
All rings and ring homomorphisms are localised, see  
Definition 
\ref{localisedrings}. 
\par
$\mathit{Hypothesis}$ 2. $\Lambda$ is a \emph{local product} of $\Lambda_{1}$ and $\Lambda_{2}$, i.e. for any pair of elements
$\lambda_{1} \in \Lambda_{1, \mu}$  and  $\lambda_{2} \in \Lambda_{2, \mu}$ such that $j_{1}(\lambda_{1, \mu}) = j_{2}(\lambda_{2, \mu}) = \lambda'  \in \Lambda'_{\mu}$, there exists
only one element $\lambda_{n} \in \Lambda_{\mu}$ such that  $i_{1}(\lambda_{\mu}) = \lambda_{1, \mu}$ and 
$i_{2}(\lambda_{\mu}) = \lambda_{2, \mu}$.
\par
The ring structure in $\Lambda$ is defined by 
\begin{equation}   
(\lambda_{1}, \lambda_{2}) + (\lambda_{1}', \lambda_{2}') := (\lambda_{1} + \lambda_{1}' \;,\; \lambda_{2} + \lambda_{2}'), \;\;
(\lambda_{1}, \lambda_{2}) . (\lambda_{1}', \lambda_{2}') := (\lambda_{1} . \lambda_{1}' \;,\; \lambda_{2} . \lambda_{2}'), 
\end{equation}  
i.e. the ring operations in $\Lambda$ are performed component-wise.
\par
$\mathit{Hypothesis}$ 3. At least one of the homomorphisms $j_{1}$  and $j_{2}$ is surjective.
\begin{remark}   
-i) Any matrix $M \in \mathbb{M}_{n}(\Lambda)$ consists of a pair of matrices 
$(M_{1}, M_{2}) \in \mathbb{M}_{n}(\Lambda_{1}) \times \mathbb{M}_{n}(\Lambda_{2}) $ subject to the condition
$j_{1, \ast} M_{1} = j_{2, \ast} M_{2}$.   
 Any matrix $M \in \mathbb{M}_{n}(\Lambda)$ is called \emph{double} matrix.
 \par
 -ii) if  $(M_{1},  M_{2})$,   $(N_{1},  N_{2})$ are double matrices, ({\bf resp.} belong to $\mathbb{M}(\Lambda_{1})  \times \mathbb{M}(\Lambda_{2})$)
 and $(\lambda_{1}, \lambda_{2}) \in   \Lambda$, ({\bf resp.}  $(\lambda_{1}, \lambda_{2})  \in  \Lambda_{1} \times \Lambda_{2}$ )  then  relations (21.6)   induce onto the space of double matrices, ({\bf resp.} the space $\mathbb{M}(\Lambda_{1})  \times  \mathbb{M}(\Lambda_{2})$)
 the following relations
 $$
(\lambda_{1}, \lambda_{2}) \; (M_{1},  M_{2}) =   (\lambda_{1} \; M_{1},  \lambda_{2} \; M_{2})
 $$
 $$
(M_{1},  M_{2}) \; +  \; (N_{1},  N_{2})  =   (M_{1} + N_{1}, \; M_{2} + N_{2})
 $$
$$
(M_{1},  M_{2}) \; .  \; (N_{1},  N_{2})  =   (M_{1} \; .  \;N_{1}, \; M_{2} \;.  \;N_{2})
 $$
 \end{remark}   
\begin{definition} 
A commutative diagram satisfying Hypotheses 1. 2. 3. will be called \emph{local Mayer-Vietoris diagram}.
\end{definition}   
\emph{Standing Hypothesis}.  In the remaining part of this chapter we assume that the ring
$\mathcal{A}$ is localised. We assume also that $\Lambda_{1} = \Lambda_{2} = \mathcal{A}$
and that $J \subset \mathcal{A}$ is a bi-lateral ideal. Define $\Lambda := \{  (\lambda_{1}, \lambda_{2}) \in \mathcal{A} \oplus \mathcal{A}, \;such \; that \; \lambda_{1} - \lambda_{2} \in J   \}$. Denote $i_{\alpha} (\lambda_{1}, \lambda_{2}) = \lambda_{\alpha}$,  $\alpha = 1, or \;2$.
Denote also    $\Lambda^{'} :=  \mathcal{A}/ J$ and $j_{\alpha}(\lambda_{1}, \lambda_{2}) :=
\lambda_{\alpha} \; mod. \; J$. One has the ring short exact sequence
$$
 0 \longrightarrow \Lambda 
  \overset{(i_{1}, i_{2})}\longrightarrow 
 \Lambda_{1} \oplus \Lambda_{2}
 \overset{j_{1} - j_{2}}\longrightarrow 
 \Lambda^{'} \longrightarrow 0.
$$
We assume that the above scheme is a localised Mayer - Vietoris diagramme.

\section{Preparing the definition of $T_{0}^{loc}(\mathcal{A})$ and  $T_{1}^{loc}(\mathcal{A})$.}           
\begin{definition}  
 We assume the ring  $\mathcal{A}$ is localised.
 \par
We consider the space of  \emph{matrices with entries in}  $\mathcal{A}_{\mu}$ and we  denote it by
 $\mathbb{M}_{n}(\mathcal{A}_{\mu})$.
\par
Let $ \mathbb{I}demp_{n}(\mathcal{A}_{\mu}) $  
denote the space of idempotent matrices of size $n$ with entries in $\mathcal{A}_{\mu}$.
\par
Let $ \mathbb{GL}_{n}(\mathcal{A}_{\mu}) $  
denote the space of invertible matrices $M$ of size $n$ with the property that the entries of both $M$ and $M^{-1}$ belong to  $\mathcal{A}_{\mu}$. 
\par 
Let  $\mathbb{I}demp  (\mathcal{A}_{\mu}) := \injlim_{n} \mathbb{I}demp_{n} (\mathcal{A}_{\mu})$. 
\par
Let $\mathbb{GL}  (\mathcal{A}_{\mu}) := \injlim_{n} \mathbb{GL}_{n} (\mathcal{A}_{\mu})$. 
\end{definition}   
\par 

\begin{definition}  
-i) Two matrices $s, \; t \in \mathbb{M}_{n}(\mathcal{A})$ are called \emph{conjugated} and we write $s \sim_{c} t$, provided they are related through an inner auto-morphism, i.e. there exists $u, \; u^{-1} \in \mathbb{GL}_{n}(\mathcal{A})$ such that  
$s = u t u^{-1}$.
\par
-ii) Two matrices $s, \; t \in \mathbb{M}_{n}(\mathcal{A}_{\mu})$ will be called \emph{locally conjugated}  and we write $s \sim_{lc} t$ provided they are related through a \emph{local} inner auto-morphism defined by 
$u, u^{-1} \in \mathbb{GL}_{n}(\mathcal{A}_{\mu})$ such that  $s = u t u^{-1}$.
\par
In particular, 
\par
-ii.1)
two idempotents $p, \; q \; \in \mathbb{I}demp_{n}(\mathcal{A}_{\mu})$ are  \emph{locally conjugated} and we write $p \sim_{lc} q$  provided there exists $ u, u^{-1} \in \mathbb{GL}_{n}(\mathcal{A}_{\mu})$ such that   
$q = u \; p \; u^{-1}$ and
 \par
 -ii.2)
 two invertible matrices $s, \; t \in \mathbb{GL}_{n}(\mathcal{A}_{\mu})$ are \emph{locally conjugated} and we write $s \sim_{lc} t $ provided there exists $u, u^{-1} \in \mathbb{GL}_{n}(\mathcal{A}_{\mu})$ such that  
 $s = u t u^{-1}$.
 \par
 Sometimes, if it is clear from the context,  $\sim_{lc}$ will be simply denoted by $\sim_{c}$.
\end{definition}  
\begin{proposition}   
\par
-i)
 $\mathbb{I}demp_{n}(\mathcal{A}_{\mu})$ and  $\mathbb{GL}_{n}(\mathcal{A}_{\mu})$
 are semigroups with respect to the direct sum
 \begin{gather}  
 A + B :=
 \begin{pmatrix}
 A & 0\\
 0 & B
 \end{pmatrix}
 \end{gather}   
 \par
 -ii) The operations of stabilisation, conjugation and local conjugation of idempotents, resp. invertible matrices, commute.
\end{proposition}  
\par
 The spaces  $\mathbb{I}demp_{n}(\mathcal{A}) $,  $\mathbb{GL}_{n} (\mathcal{A}_{\mu}) $ are compatible with stabilisations.
 \par
 The direct sum addition of idempotents, resp. invertibles,  is compatible with the local conjugation equivalence relation. Indeed, if  $s_{1}, {s_{2}}$ are conjugated through an inner auto-morphism defined by the element $u_{1}$  ($s_{1} \sim_{l} {s_{2}}$) and
 $t_{1}, {t_{2}}$ are conjugated through an inner auto-morphism defined by the element $u_{2}$  ($t_{1} \sim_{l} {t_{2}}$), then 
 $(s_{1} + t_{1})   \sim_{l} (s_{2} + t_{2})$  are conjugated through the inner auto-morphism  $u_{1} \oplus u_{2}$. 
  \par
  With this observation, the associativity of the addition is now immediate.
  \par
These show that
  $ \mathbb{I}demp(\mathcal{A}_{\mu}) \; / \sim_{lc}$, resp.  $ \mathbb{GL}(\mathcal{A}_{\mu}) \;/ \sim_{l}$,
 is an associative semi-group.
\begin{proposition}   
The semi-groups 
 $ \mathbb{I}demp(\mathcal{A}) \; / \sim_{c}$,   $ \mathbb{GL}(\mathcal{A}_{\mu}) \;/ \sim_{lc}$ are commutative.
\end{proposition}  
\begin{proof}  
The result follows from the following identity valid for any two matrices
$A, B \in  \mathbb{M}_{n}( \mathcal{A}_{\mu}) $ 
\begin{equation}   
\begin{pmatrix}
A & 0\\
0 & B
\end{pmatrix}
=
\begin{pmatrix}
0 & -1\\
1 & 0
\end{pmatrix}
\begin{pmatrix}
B & 0\\
0 & A
\end{pmatrix}
\begin{pmatrix}
0 & 1\\
-1 & 0
\end{pmatrix} =
\begin{pmatrix}
0 & -1\\
1 & 0
\end{pmatrix}
\begin{pmatrix}
B & 0\\
0 & A
\end{pmatrix}
\begin{pmatrix}
0 & -1\\
1 & 0
\end{pmatrix} ^{-1}
,
\end{equation}   
which tells that  $(A + B) \sim_{c} (B+A)$, resp.  $(A + B) \sim_{lc} (B+A)$,
\end{proof}  
\par
For more information about the relationship between the classical algebraic $K$-theory and the \emph{local} $T$-theory  see \ref{21Introduction}.    

 
\section{Definition of $T_{0}(\mathcal{A}_{\mu})$ and $T_{0}^{loc}(\mathcal{A})$.}  
We suppose the stabilisation is involved without being specified. 
 \begin{definition}   
 Suppose $\mathcal{A}$ is a localised unital associative ring. We define
 \par
 -i)
  \begin{equation}  
 T_{0} (\mathcal{A}_{\mu}) \; = \;
 \mathbf{G} \; (\mathbb{I}demp_{n}\; (\mathcal{A}_{\mu}) \;/ \sim_{c}).
 \end{equation}    
\par
 -ii) 
 \begin{equation}  
 T_{0} (\mathcal{A}) \; = \; ProjLim_{\mu \in \mathbb{N}} \; T_{0} (\mathcal{A}_{\mu}). 
 \end{equation}    
 where $\mathbf{G}$ is \emph{Grothendieck completion}.
\end{definition}
\par
Any local inner automorphism induces the identity on $T_{0}^{loc} (\mathcal{A})$.
\begin{proposition}  
For any unital associative algebra $\mathcal{A}$,  trivially localised ($\mathcal{A}_{\mu} = \mathcal{A}$),  $\;T_{0}(\mathcal{A}) = K_{0}(\mathcal{A})$.
\end{proposition}  
\begin{proof}
See Rosenberg \cite{Rosenberg}  Lemma 1.2.1.
\end{proof}
\section{ $T_{1}(\mathcal{A}_{\mu})$  and  $T_{1}^{loc} (\mathcal{A})$.}  
\par
As in the previous sub-section, we assume the stabilisation is involved without being specified.
\par
The equivalence class of the invertible element $u \in \mathbb{GL}(\mathcal{A})$ modulo conjugation, $[u]_{\sim_{c}}$, will be called the \emph{abstract Jordan canonical form} of $u$. The group
$T_{1}(\mathcal{A}_{\mu})$ we are going to define preserves much of the information provided by the abstract Jordan form. The classical definition of $K_{1}$ extracts a minimal part of the abstract Jordan form.  As the addition in the semi-group $\mathbb{GL}(\mathcal{A})$ is given by  the direct sum and the Jordan canonical form $J(u)$  (in the classical case of the algebra $\mathbb{GL}_{n}(\mathbb{R})$) $J$ behaves additively ($J(u\oplus v) = J(u) \oplus J(V)$\; modulo permutations of the Jordan blocks), given an arbitrary element  $u \in \mathbb{GL}(\mathcal{A})$, it is not reasonable to expect existence of an element $\tilde{u}$ such that $[u + \tilde{u}]_{\sim_{c}} = [1_{2n}]_{\sim_{c}}$. Given that we want  $T_{1}(\mathcal{A}_{\mu})$ to be a group, we introduce the group structure (opposite elements) forcibly. In the case of the classical $K_{1}$-theory, the class of the element $u^{-1}$ represents the opposite class,  $ - [u] \in K_{1}(\mathcal{A})$.
In our case, the opposite elements will be introduced by means of the \emph{T-completion} technique, to be explained next.
\section{$T$-completion}  \label{Tcompletion}
\begin{proposition}     
Let $\mathcal{S}$ be an additive commutative semi-group with zero element $0$. Let $I: \mathcal{S} \longrightarrow \mathcal{S}$ be an  \emph{additive involution}, ($I^{2} = Id$), such that $I(0) = 0$.
\par
Define the equivalence relation $\sim_{\mathcal{O}}$ in $\mathcal{S}$:   $u \sim_{\mathcal{O}} v$   iff these exist two
elements  $u_{0}, u_{1} \in \mathcal{S}$, such that the elements $u, v$, $\xi_{0} = u_{0} + I (u_{0})$ and 
$\xi_{1} = u_{1} + I (u_{1})$,   satisfy
\begin{equation}  
u + \xi_{0} = v +  \xi_{1}.
\end{equation}  
-i) The equivalence relation $\sim_{\mathcal{O}}$ is compatible with the addition in $\mathcal{S}$.
\par
-ii) $\mathcal{S}/ \sim_{\mathcal{O}}$ is a group.  
\par
-iii) Let  $[u]$  denote the $\sim_{\mathcal{O}}$-equivalence class of  $u$.
Then 
\begin{equation}  
- [u] = [I(u)].
\end{equation}  
\end{proposition}  
\begin{definition}  
Define
\begin{equation}  
\mathcal{O}(u) := u + I(u) \;\;for \;any \;u \in \mathcal{S}
\end{equation}  
\begin{equation}  
O (\mathcal{S}):= \; \{ \mathcal{O}(u) \;|\; u \in \mathcal{S} \;\}
\end{equation}  
\end{definition}   
\par
\begin{proof}   
-i) Obvious.
\par
-ii) $\sim$ is clearly reflexive and symmetric.
 If $u_{1} + \xi_{1} = u_{2} + \xi_{2}$ and $u_{2} + \xi_{3} = u_{3} + \xi_{4}$,  $u_{1}, u_{2}, u_{3}, u_{4} \in O(\mathcal{S})$,
then $u_{1} + (\xi_{1} + \xi_{3}) = u_{2} + (\xi_{2} + \xi_{3}) = u_{3} +  (\xi_{2} + \xi_{4})$, which shows that $\sim$ is transitive. Therefore $\sim$ is an equivalence relation.
\par
-iii) One the other side, for any element $u \in \mathcal{S}$, one has
$[u] = [u + 0] = [u] + [0] =  [0] + [u] $,
which shows that the class of $[0]$ is a the zero element of $\mathcal{S} / \sim_{\mathcal{O}}$.
\par
We have also $[u] + [I(u)] = [\mathcal{O}(u)] = 
[0 + \mathcal{O}(u)] =  [0]$ because $0  \sim (0 + \mathcal{O}(u))$; therefore, $-[u] = [I(u)]$ exists.
From this we get further $[\mathcal{O}(u)] = [u + I(u)] =  [u] + [I(u)] = 0$.
\end{proof}   
\section{Definition of $T_{1}(\mathcal{A}_{\mu})$  and  $T_{1}^{loc} (\mathcal{A})$} 
The construction of the groups  $T_{1}(\mathcal{A}_{\mu})$  and  $T_{1}^{loc} (\mathcal{A})$ 
involves the $T$-completion    \S \ref{Tcompletion}. We take
\begin{itemize}
\item
$\mathcal{S} = \mathbb{GL}(\mathcal{A}_{\mu}) \sim_{c} $
\item
the involution $I:  \mathbb{GL}(\mathcal{A}_{\mu}) \sim_{cl} \longrightarrow \mathbb{GL}(\mathcal{A}_{\mu}) \sim_{cl}$ is given by $I(u) = u^{-1}$ 
\item
the identity element $I \in \mathbb{GL}(\mathcal{A}_{\mu})/ \sim_{c}$ becomes the zero element of $\mathcal{S}$.
\end{itemize}
\par
Here $\sim_{l}$ indicates that the invertible elements and their conjugation occurs locally.
Note that if $u_{1} \sim_{cl} u_{2}$, then  $u^{-1}_{1} \sim_{cl} u^{-1}_{2}$.
\par
The next proposition summarises the properties of $\mathcal{O}(\mathcal{A}_{\mu})$ 
\begin{proposition}   
-i) The space $\mathcal{O}(\mathcal{A}_{\mu})$ 
is a commutative group; the zero element is the class if the identity
\par
-ii) the mapping
\begin{equation}  
\mathcal{O}: \mathbb{GL}(\mathcal{A}_{\mu})  \longrightarrow \mathcal{O}(\mathcal{A}_{\mu} )
\end{equation}  
is additive and commutes with \emph{local} conjugation
\begin{equation}  
\mathcal{O}(u_{1} + u_{2} )  =   \mathcal{O}( u_{1} )  +     \mathcal{O}( u_{2} )  
\end{equation}  
\begin{equation}  
\mathcal{O}(\lambda \; u \; \lambda^{-1})  =   \lambda  \; \mathcal{O}( u )   \; \lambda^{-1}, \;\; \lambda \in \mathbb{GL}(\mathcal{A}_{\mu}),
\end{equation}  
i.e. if $u_{1} \sim_{sl} u_{2}$ then  $\mathcal{O}(u_{1}) \sim_{sl} \mathcal{O}(u_{2})$ 
\par
-iii) 
\begin{equation}   
\mathcal{O}(u^{-1}) \sim_{c}  \mathcal{O}(u).
\end{equation}   
\par
-iv)
\begin{equation}  
\mathcal{O}(u_{1} u_{2}) \neq   \mathcal{O}(u_{1}) \mathcal{O}(u_{2}),
\end{equation}  
\end{proposition}     
\begin{definition}  
\begin{equation}   
T(\mathcal{A}_{\mu}) :=  T-completion \;of \;\mathbb{GL}(\mathcal{A}_{\mu})/ \sim_{lc}
\end{equation}  
and
\begin{equation}   
T^{loc}(\mathcal{A}) :=  \projlim_{\mu \in \mathbb{N}} T_{1}(\mathcal{A}_{\mu})
\end{equation}    
\end{definition}  
\begin{proposition}   
-i)
$T_{1}(\mathcal{A}_{\mu})$  and $T_{1}^{loc}(\mathcal{A})$ are commutative groups.
\par
-ii) The inverse element of $[u]$ is  $-[u] = [u^{-1}]$.
\end{proposition}  
\begin{proof}  
-i) and -ii) follow from the properties of the T-completion, see Proposition 21.5  \S \ref{Tcompletion}.  
\end{proof}   
 \vspace{1cm}
 For the computation of the  {\em local algebraic $T$-theory}, $ i = 0, 1$, of the algebra of complex numbers $\mathbb{C}$
see Teleman \cite{Teleman_arXiv_V}.
\par
For the computation of the \emph{local} cyclic homology of the algebra 
of Hilbert-Schmidt operators on simplicial spaces see  \cite{Teleman_arXiv_II}.
\par
For the local index theorem see \cite{Teleman_arXiv_IV},   \cite{Teleman_arXiv_I}. 

\section{Induced homomorphisms.}   
 \begin{definition}   
Let $f:  \mathcal{A} \longrightarrow  \mathcal{B}$  be a localised ring homomorphism.
\par
 Then the ring homomorphism $f$  induces homomorphisms
\begin{equation}   
f_{\ast}:  T_{0} (\mathcal{A}_{\mu})   \longrightarrow  T_{0} (\mathcal{B}_{\mu}),  \hspace{0.5cm}
f_{\ast}:  T_{0}^{loc} (\mathcal{A})  \longrightarrow  T_{0}^{loc} (\mathcal{B})
\end{equation}  
and
\begin{equation}   
f_{\ast}:  T_{0} (\mathcal{A}_{\mu})   \longrightarrow  T_{0} (\mathcal{B}_{\mu}),  \hspace{0.5cm}
f_{\ast}:  T_{1}^{loc} (\mathcal{A})  \longrightarrow  T_{1}^{loc} (\mathcal{B})
\end{equation}  
 \end{definition}   
\section{Constructing idempotents and invertible matrices over 
$\Lambda_{\mu}$.}  
We come back to the situation presented in the section 
\ref{MayerVietoris}.   
Here we produce \emph{local idempotents and invertibles} of the ring $\Lambda$.  These results will be used in the proof of the \emph{six terms exact sequence Theorem} \ref{exactsequence}.  
\subsection{Constructing idempotents over $\Lambda_{\mu}$.}   
\begin{theorem}  \label{Theorem1}   
Let $p_{1} \in \mathbb{I}demp_{n}(\Lambda_{1, \mu})$ be  an idempotent matrix with entries in $\Lambda_{1, \mu}$, resp. 
$p_{2} \in \mathbb{I}demp_{n}(\Lambda_{2, \mu})$, 
 such that
\begin{equation}  
j_{1\ast} (p_{1}) =  u \;  j_{2\ast} (p_{2}) \; u^{-1},
\end{equation}  
where $u \in \mathbb{GL}_{n}(\Lambda'_{\mu})$ is an invertible matrix.
\par
-i) Then there exists an idempotent double matrix $p \in \mathbb{I}demp(\Lambda_{\mu})$ such that
\begin{equation}  
i_{1\ast} (p) = p_{1} \oplus 0_{n}  \;\;  and  \;\; \; i_{2\ast} (p) = \tilde{p} 
\end{equation}  
where the idempotent $\tilde{p}_{2} \in \mathbb{M}_{2n}(\Lambda_{2, \mu})$ is conjugated to 
$p_{2} \oplus 0_{n}$ through an invertible matrix $\tilde{U}(u) \in \mathbb{GL}_{2n} (\Lambda_{2, \mu})$, that is  
\begin{gather}  
\tilde{p}_{2} = \tilde{U}(u) (p_{2} \oplus 0_{n}) \tilde{U}(u)^{-1}  \\
j_{2, \ast} (\tilde{p}_{2}) = (u \; j_{1, \ast} (p_{1}) \; u^{-1}) \oplus 0_{n} = j_{1, \ast} ({p}_{1} \oplus 0_{n})
\end{gather}   
-ii) 
The corresponding double matrix idempotent is denoted $p =(p_{1}, p_{2}, \tilde{U}(u))$. 
\end{theorem} 
\begin{definition} 
Let $\tilde{U}(u)$ be the lifting 
(\ref{factorisation}) 
of $u \otimes u^{-1}$  in $\Lambda_{2,\mu}$. 
\par
Denote by $p = (p_{1}, p_{2}, \tilde{U}(u))$ the idempotent over $\Lambda_{\mu}$ produced by Theorem \ref{Theorem1}.  
\end{definition}  
\par
Condition 
(50)   
says that $[j_{1 \ast }p_{1} ] = [j_{ \ast }p_{2} ] \in T_{0} (\Lambda'_{\mu})$.
Part -i) says that the pair  $([p_{1}\oplus 1_{n}], [\tilde{p}_{2}]) \in T_{0}(\Lambda_{1})  \oplus T_{0}(\Lambda_{2}) $ belongs to the image of
$(i_{1 \ast},  i_{2 \ast})$. 
\par
\begin{proof}  
To prove this theorem we will operate onto objects related to $\Lambda_{2, \mu}$.
\begin{lemma}  
Let $p_{1} = ( a_{ij} ) \in \mathbb{I}demp_{n} (\Lambda_{1, \mu})$  and $p_{2} = ( b_{ij} ) \in \mathbb{I}demp_{n} (\Lambda_{2, \mu})$ be idempotents.
\par
Suppose the idempotents $ j_{1\ast} (p_{1})$,  $ j_{2\ast} (p_{2}) $ are conjugated through an inner automorphism defined by
$u \in \mathbb{GL}_{n} (\Lambda'_{\mu})$, i.e.
\begin{equation}   
j_{1\ast} (p_{1}) =  u \;  j_{2\ast} (p_{2}) \; u^{-1}.
\end{equation}  
Assume, additionally, that the invertible element $u$ lifts to an invertible element 
$\tilde{u} \in \mathbb{GL}_{n}(\Lambda_{2, \mu})$ 
(i.e. $j_{2\ast} \tilde{u} = u$).
\par
Then $p =(p_{1}, p^{'}_{2}, u)   \in  \mathbb{I}demp_{n} (\Lambda_{\mu})$ is an idempotent given by the double matrix
\begin{equation}  
p = ( (a_{ij}, c_{ij})) ,
\end{equation}   
where
\begin{equation}   
 (a_{ij}) = p_{1} \in \mathbf{I}demp_{n}(\Lambda_{1, \mu}) \;\; and \;\; p^{'}_{2} = ( c_{ij} )  := 
 \tilde{u} \;   p_{2} \; \tilde{u}^{-1} \in \mathbb{I}demp_{n}  ( \Lambda_{2,  \mu}).
\end{equation}   
\end{lemma}   
remark that in this lemma the size of the double matrix $p$ does not change. 
\par
$\mathit{Proof \; of \; Lemma} \; 35$.  
We use Remark 19-ii).
It is clear that the matrix $p$ given by 
(55) 
is an idempotent. In fact, to evaluate $p^{2}$ amounts to compute separately  the square of the first and second component matrices of the matrix $p$,  i.e. the squares of $(a_{ij})$ and $(c_{ij})$. These are 
\begin{equation}   
(a_{ij})^{2} = (a_{ij})  \;\;and\;\;
(c_{ij})^{2} =  (\; \tilde{u}  \; p_{2} \; \tilde{u}^{-1} \;)^{2} =  \tilde{u}  \; p_{2}^{2} \; \tilde{u}^{-1} = \tilde{u}  \; p_{2} \; \tilde{u}^{-1} = (c_{ij}).
\end{equation}   
\par
It remains to verify that  $p \in \mathbb{M}_{n} (\Lambda_{\mu})$, i.e.
$j_{1 \ast} (a_{ij}) = j_{2 \ast} (c_{ij}) $.  
This follows from (40) combined with (54)
\begin{equation}  
j_{1\ast} (p_{1}) =  u \;  j_{2\ast} (p_{2}) \; u^{-1} = j_{2 \ast} (\tilde{u} \;  p_{2} \; \tilde{u}^{-1} ) =  j_{2 \ast} (\tilde{p}_{2}).  
\end{equation}  
This ends the proof of Lemma 35.
\begin{lemma}  
Let $p_{1} = ( a_{ij} ) \in \mathbb{I}demp_{n} (\Lambda_{1, \mu})$  and $p_{2} = ( b_{ij} ) \in \mathbb{I}demp_{n} (\Lambda_{2, \mu})$ be idempotents.
\par
Suppose the idempotents $ j_{1\ast} (p_{1})$,  $ j_{2\ast} (p_{2}) $ are conjugated through an inner automorphism defined by
$u \in \mathbb{GL}_{n} (\Lambda^{'} _{\mu})$, i.e.
\begin{equation}   
j_{1\ast} (p_{1}) =  u \;  j_{2\ast} (p_{2}) \; u^{-1}.
\end{equation}  
Then
\par
-i) $ j_{1 \ast}(p_{1} \oplus 0_{n}) $ and $ j_{2 \ast}(p_{2} \oplus 0_{n})$ are conjugated by 
$U := u \oplus u^{-1}  \in  \mathbf{GL}_{2n}(\Lambda'_{\mu}) $,  i.e.
\begin{equation}  
 j_{1 \ast}(p_{1}) \oplus 0_{n} =
 j_{1 \ast}(p_{1} \oplus 0_{n})  = U \;  j_{2 \ast}(p_{2} \oplus 0_{n}) \; U^{-1} =  (\; u  \; j_{2 \ast}(p_{2}) \;  u^{-1} \; ) \oplus 0_{n}.
\end{equation}  
\par
-ii) Supposing that $j_{2}$ is surjective, the invertible matrix $U$ lifts to an invertible matrix $\tilde{U} \in \mathbb{M}_{2n} (\Lambda_{2, \mu})$. Let 
\begin{equation}  
\tilde{p}_{2} := \tilde{U}   \; (p_{2} \oplus 0_{n})  \;  \tilde{U}^{-1}.
\end{equation}    
Then
\begin{gather} 
j_{1, \ast} (p_{1} \oplus 0_{n}) = (u \; j_{2, \ast} ({p}_{2}) \; u^{-1})\oplus 0_{n} = j_{2, \ast} (\tilde{p}_{2} );
\end{gather}  
i.e.  the matrices $p_{1} \oplus 0_{n}, \; \tilde{p}_{2}$ form a double matrix idempotent in 
$\mathbb{I}demp_{2n}(\Lambda_{\mu})$, denoted $p := (p_{1}, p_{2}, u)  \in  \mathbb{I}demp_{2n}(\Lambda_{\mu})  $ and
 \begin{equation}  
 (i_{1 \ast}, i_{2 \ast} ) p = (p_{1}, \tilde{p}_{2}) 
 \end{equation}  
\par
-iii) The pair of idempotents  $p_{1}  \in \mathbb{I}demp_{n} (\Lambda_{1, \mu})$,   $p_{2}  \in \mathbb{I}demp_{n} (\Lambda_{2, \mu})$  is stably equivalent to the pair of idempotents $p_{1} \oplus 0_{n}$,  ${p}_{2} \oplus 0_{n}$ and ${p}_{2} \oplus 0_{n} \sim_{l}  \tilde{p}_{2}$. In other words
 \begin{equation}  
     ([p_{1}], [{p}_{2}]) =  ([p_{1}], [\tilde{p}_{2}])   \in T_{0} (\Lambda_{1, \mu }) \oplus T_{0}(\Lambda_{2}, \mu) .
 \end{equation}   
 \end{lemma}   
 Note that in this lemma, in comparison with the preceding 
 Lemma 35, 
 the size of the desired idempotent doubles; otherwise, the important modifications still occur onto matrices associated with  $\Lambda_{2, \mu}$. 
 \par
 $\mathit{Proof \; of \; Lemma \; 36}$.   
 Part -i) is clear.
 \par
 The proof of -ii) uses $\mathcal{O}_{2n}(u)$, where $u \in \mathbb{GL}_{n}(\mathcal{A}_{\mu})$. 
 Recall that 
 $\mathcal{O}_{2n}(u)$ may be written as a product of \emph{elementary matrices} and a \emph{scalar matrix} (\ref{factorisation})  
\begin{equation}  
U := \mathcal{O}_{2n}(u) = 
\begin{pmatrix}
u & 0 \\
0 & u^{-1}
\end{pmatrix}
=
\begin{pmatrix}
1 & u \\
0 & 1
\end{pmatrix}
\begin{pmatrix}
1 & 0 \\
- u^{-1} & 1
\end{pmatrix}
\begin{pmatrix}
1 & u \\
0 & 1
\end{pmatrix}
\begin{pmatrix}
0 & -1 \\
1 & 0
\end{pmatrix}
\in \mathbf{M}_{2n}(\Lambda_{2, \mu}).
\end{equation}    
Having in mind the localisation of the ring $\mathcal{A}$, it is important to note that the entries of the formula 
(65)  
depend on $u$ and $u^{-1}$ only.
\par
The proof will be complete after we will have shown that the invertible matrix $U$ has an invertible lifting $\tilde{U} \in \Lambda_{2, \mu}$. This follows from the properties of elementary matrices (valid for matrices and block matrices) to be discussed next.

\begin{definition}  \emph{Elementary Matrices}. 
\par
A matrix $E_{ij} (a) \in \mathit{GL}_{n}(\mathcal{A}_{\mu}) $ having all entries equal to zero, except for the diagonal entries equal to $1$ and just one $(i,j)$-entry $a \in \mathcal{A}_{\mu}$,  $0 \leq i \neq j \leq n$  is called  \emph{elementary matrix with entry in} $\mathcal{A}_{\mu}$. 
\par
The space of \emph{ elementary matrices with entries in} $\mathcal{A}_{\mu}$ is by definition 
\begin{equation}  
\mathit{E}_{n} (\mathcal{A}_{\mu}) \; := \{  E_{ij} (a) \; | \; 1 \leq i \neq j \leq n, \; a \in \mathcal{A}_{\mu} \}.
\end{equation}   
Let $\mathbf{E}_{n} (\mathcal{A})$ the sub-group generated by all elementary matrices and
\begin{equation}   
\mathbf{E}(\mathcal{A}) := \injlim_{n \in \mathbf{N}} \mathbf{E}_{n} (\mathcal{A}).
\end{equation}    
\end{definition}  
\par
\begin{lemma}   
Begin
The elementary matrices satisfy
\begin{equation}   
E_{i,j}(a) . E_{i,j}(b)  \; = \; E_{i,j}(a+b)  
\end{equation}  
\begin{equation}  
E_{i,j}(a)^{-1}  \; = \; E_{i,j}(-a),   \label{inverse}            
\end{equation}  
therefore $E_{i,j}(a)  \in \mathbb{GL} (\mathcal{A}_{\mu})$.
 Any elementary matrix is a commutator
\begin{equation}  \label{commutator}    
E_{ij} (A) = [ E_{ik}(A), E_{kj}(1)],   \hspace{0.2cm}  for \; any\; i, j, k \; distinct \; indices.
\end{equation}  
\end{lemma}   
\par
We come back to the proof of 
Lemma  36 -ii).   
As $j_{2}$ is surjective, each of the entries of factors of the RHS of 
(65)  
has a lifting in $\mathbf{M}_{2n}(\Lambda_{2 \mu})$; each of the elementary matrix factor lifts as an invertible elementary matrix.  The last factor lifts as it is. 
Therefore, $U$ has an invertible lifting $\tilde{U} \in \mathbb{GL}_{2n}(\Lambda_{2, \mu})$. 
Lemma 35 
completes the proof of 
Lemma 36  -ii).  
\par
Lemma 36 -iii)  
follows from the definition of $T^{loc}_{0}(\Lambda)$.
This completes the proof of Lemma 21.2.  
\par
Theorem 33 
follows from Lemma  34 
combined with Lemma 35.  
\end{proof}   

\par
Theorem 33 
refers to the construction and description of idempotents over $\Lambda_{\mu}$. 
We need to extend Theorem 33 to elements of $T_{0}^{loc}(\mathcal{A}),$  i.e. to formal differences of \emph{local} idempotents. 
\par
\begin{lemma}    (compare  \cite{Milnor_Introduction}  Lemma 1.1)  
\par
Let $p_{1}, p_{2} \; q_{1}, q_{2} \in \mathbf{I}demp_{n}(\mathcal{A}_{\mu})$ be idempotents and let $[\;\;]$ denote their $T_{0}(\mathcal{A}_{\mu})$ class.
Suppose
\begin{equation}  
[p_{1}] - [p_{2}] \;=\; [q_{1}] - [q_{2}] \in T_{0}(\mathcal{A}_{\mu}).
\end{equation}  
\par
Then $p_{1} + q_{2}$ and $p_{2} + q_{1}$ are \emph{locally, stably} isomorphic,  $p_{1} + q_{2} \sim_{ls} p_{2} + q_{1}$.
\end{lemma}   
\begin{proof} 
The stabilisation and Grothendieck completion imply that there exists an idempotent 
$s \in  \mathbb{I}demp_{m}(\mathcal{A}_{\mu})$ such that the idempotents 
$$p_{1} + q_{2} + s,   \;\;\;\;  p_{2} + q_{1}+ s$$ 
are \emph{locally, stably} isomorphic. We assume that the idempotent $s$ is already sufficiently stabilised.
This means there exists an invertible matrix $u \in \mathbb{GL}_{2n+m}(\mathcal{A}_{\mu})$ 
such that
$$p_{1} + q_{2} + s,   \; = \; u \; ( p_{2} + q_{1}+ s ) \; u^{-1}.$$ 
We add to both sides of this equality the idempotent $1_{m} - s$ and we extend $u$ to be the identity on the last summand. 
We get  
$$p_{1} + q_{2} + s + (1-s) ,   \; = \; u \; ( p_{2} + q_{1}+ s +(1-s) ) \; u^{-1}.$$ 

From this we get further
$$p_{1} + q_{2} + 1_{2m} ,   \; = \; u \; ( p_{2} + q_{1} + 1_{2m} ) \; u^{-1},$$ 
that is, the idempotents $p_{1} + q_{2} $,     $p_{2} + q_{1} $ are \emph{locally, stably} isomorphic
$$p_{1} + q_{2} \; \sim_{sl}    p_{2} + q_{1}. $$ 
\end{proof}   
\begin{lemma}  
Let $p, q \in \mathbb{I}demp_{n}(\mathcal{A}_{\mu})$ be idempotents.
Suppose
\begin{equation}  
[p] - [1_{n}] \;=\; [q] - [1_{n}] \in T_{0}(\mathcal{A}_{\mu}).
\end{equation}  
Then
\par
-i)  $p$ and $q$ are \emph{locally, stably} isomorphic,  $p \sim_{ls} q$,
\par
-ii)  there exists an $N \in \mathbb{N}$ and an $u \in \mathbb{GL}_{n+N}(\mathcal{A}_{\mu})$ such that
\begin{equation}  
p + 1_{N} = u \; q \; u^{-1} + 1_{N} =  u \; (q + 1_{N})\; u^{-1}.
\end{equation}  
\end{lemma}   
\begin{proof}   
-i) Lemma 39  says that  the idempotents $p + 1_{n}$,  $q + 1_{n}$ are \emph{locally, stably} isomorphic. This means that the idempotents $p$ and  $q$ are \emph{locally, stably} isomorphic. Part -ii) tells precisely this.
\end{proof}  

\begin{theorem}   \label{theorem21.2}  
Let $p_{ij}$ be idempotents 
$$[p_{1}] = [p_{11}] - [p_{12}]       \in T_{0}(\Lambda_{1, \mu})$$  
$$[p_{2}] = [p_{21}] - [p_{22}]      \in T_{0}(\Lambda_{2, \mu})$$
\par
with the property that
$$j_{1 \ast} [p_{1}] = j_{2 \ast} [p_{2}] \in T_{0}(\Lambda^{'}_{\mu}).$$
\par
Then there exists $[p] = [p_{01}] - [p_{02}] \in T_{0}(\Lambda_{\mu})$ with the property that
$$i_{1 \ast} [p] = [p_{1}] \;\; and \;\; i_{2 \ast} [p] = [p_{2}].$$
\end{theorem}  
\begin{proof}   
We may describe  the two $T$-theory classes differently

$$[p_{1}] = [p_{11}] - [p_{12}] = [p_{11}+ (1 - p_{12} )] - [p_{12} + (1-p_{12})] = [p^{'}_{12}] - [ 1_{n} ] \in T_{0}(\Lambda_{1, \mu})$$  
and
$$[p_{2}] = [p_{21}] - [p_{22}]  = [p_{21}+ (1 - p_{22}) ] - [p_{22} + (1-p_{22})] = [p^{'}_{22}] - [ 1_{n} ] \in T_{0}(\Lambda_{2, \mu}).$$
\par
Then
$$j_{1 \ast} [p_{1}] =  j_{1 \ast }(  [p^{'}_{12}] - [ 1_{n} ] ) = ( j_{1 \ast }  [p^{'}_{12}]) - [ 1_{n} ] $$ 
and 
$$ j_{2 \ast} [p_{2}]  =  j_{2 \ast }(  [p^{'}_{22}] - [ 1_{n} ] ) = ( j_{2 \ast }  [p^{'}_{22}]) - [ 1_{n} ]. $$ 
The hypothesis says that 
$$ ( j_{1 \ast }  [p^{'}_{12}]) - [ 1_{n} ] = ( j_{2 \ast }  [p^{'}_{22}]) - [ 1_{n} ]. $$ 
\par
Lemma 39  
says that the idempotents $ j_{1 \ast }  [p^{'}_{12}]$    $ j_{2 \ast }  [p^{'}_{22}]$  are \emph{locally, stably} isomorphic. Now we are in the position to use Theorem 21.1. Let $u \in \mathbb{GL}(\Lambda_{2}(\mathcal{A}_{\mu}))$; consider the conjugation
\begin{equation}  
 j_{1 \ast }  (p^{'}_{12}) = u  \; j_{2 \ast }  (p^{'}_{22}) \; u^{-1}.
\end{equation}  
Theorem 21.1 provides the idempotent $$p = (j_{1 \ast}  (p^{'}_{12}), (j_{2 \ast}  (p^{'}_{2}),u). $$
The desired idempotents are
$$
p_{10} =  p = (j_{1 \ast}  (p^{'}_{12}), (j_{2 \ast}  (p^{'}_{2}),u) \in \mathbb{I}demp(\Lambda_{\mu})
$$
$$
p_{20} = 1_{N}   \in \mathbb{I}demp(\Lambda_{\mu}).
$$
\end{proof}   

\subsection{Constructing invertibles over $\Lambda_{\mu}$.}  
\par
\begin{theorem}  \label{Theorem2}   
Let $s_{1} \in \mathbb{GL}_{n}(\Lambda_{1, \mu})$,   $s_{2} \in \mathbb{GL}_{n}(\Lambda_{2, \mu})$,  be invertible matrices with entries in $\Lambda_{1, \mu}$, resp. $\Lambda_{2, \mu}$, such that
\begin{equation}  
j_{1\ast} (u_{1}) =  u \;  j_{2\ast} (u_{2}) \; u^{-1},
\end{equation}  
where $u \in \mathbb{GL}_{n}(\Lambda'_{\mu})$.
\par
-i) Then there exists an invertible matrix $s \in \mathbb{GL}_{2n}(\Lambda_{\mu})$ such that
\begin{equation}  
i_{1\ast} (s) = s_{1} \oplus 1_{n}  \;\;  and  \;\; \; i_{2 \ast} (s) = \tilde{s}_{2} 
\end{equation}  
where the invertible matrix $\tilde{s}_{2} \in \mathbb{GL}_{2n}(\Lambda_{2, \mu})$ is conjugated to $s_{2} \oplus 1_{n}$ through an inner auto-morphism defined by the invertible matrix $\tilde{U} \in \mathbb{GL}_{2n} (\Lambda_{2, \mu})$, that is  
\begin{gather} 
\tilde{s}_{2} = \tilde{U} (s_{2}\oplus 1_{n})  \tilde{U}^{-1} \\
j_{2, \ast} \tilde{s}_{2} = (u \; j_{2, \ast} (s_{2} ) \; u^{-1}) \oplus 1_{n} = j_{1, \ast} ({s}_{1} \oplus 1_{n})
\end{gather}    
\par
-ii) The corresponding invertible double matrix  $s$ is denoted $s =(s_{1}, s_{2}, \tilde{U})$. 
\end{theorem} 
\begin{proof}    
The proof of Theorem \ref{Theorem2}  
goes along the same lines as the proof of Theorem \ref{Theorem1}.
The proof of Theorem \ref{Theorem1}  
is based on the following facts: -a) operations with double matrices respect 
Remark 19 -ii),  
lifting of the invertible element $U= \mathcal{O}_{2n}(u) := u \oplus u^{-1} \in \mathbb{GL}_{2n}(\Lambda_{2, \mu})$
($u \in \mathbb{GL}_{n}(\Lambda^{'}_{2, \mu})$) by means of the factorisation of $U$ by elementary matrices, see (\ref{factorisation}),  and  -c) the fact that any inner auto-morphism keeps invariant any zero vector sub-space.  
\par
To prove Theorem \ref{Theorem2}  
we use the same arguments -a), - b), -c) with the following changes: idempotents are replaced by invertible elements and for -c) we use the fact that the inner auto-morphisms transform the mapping $1_{n}$ into itself. 
This ends the proof of the theorem.
\end{proof}
\par
The next theorem is the analogue of Theorem \ref{theorem21.2}  
 in the $T_{1}(\mathcal{A}_{\mu})$ case.
\begin{theorem}  \label{liftinginvertibles}      
Suppose $j_{1}$ and $j_{2}$ are epi-morphisms.
\par
Let $[s_{1}] \in T_{1}^{loc}(\Lambda_{1, \mu})$ and  $[s_{2}] \in T_{1}^{loc}(\Lambda_{2, \mu})$ be such that
\begin{equation}   
j_{1 \ast}  [s_{1}]  =   j_{2 \ast}  [s_{2}] \in T_{1}^{loc}(\Lambda^{'}_{\mu}).
\end{equation}  
Then there exists $[s] \in T_{1}^{loc}(\Lambda_{\mu})$ such that
\begin{equation}   
 i_{1 \ast} [ s ] = [s_{1}] \in T_{1}^{loc}(\Lambda_{1, \mu})   \;\;\;  and \;\;\;   
 i_{2 \ast} [ s ] = [s_{2}] \in T_{1}^{loc}(\Lambda_{2, \mu}). 
\end{equation}  
\end{theorem}   
\begin{proof}   
The definition of $T_{1}^{loc}(\Lambda^{'}_{\mu})$  involves an ambiguity belonging to the sub-module $\mathbf{O}(\Lambda^{'}_{\mu})$. We assume the elements 
$s_{1}$ and  $s_{1}$ are sufficiently stabilised.  Equality 
(79)  tells there exist two elements 
$\xi_{1}, \xi_{2} \in \mathbf{O}_{2n}(\Lambda^{'}_{\mu})$ such that the invertible matrices
$j_{1 \ast}(s_{1}) +\xi_{1}$, $j_{2 \ast}(s_{2}) + \xi_{2}$ are \emph{locally conjugated} by means of a matrix   $u \in \mathbf{GL}_{2n}(\Lambda^{'})$
\begin{equation}  
j_{1 \ast}(s_{1}) + \xi_{1} = u \; (j_{2 \ast}(s_{2}) + \xi_{2}) \;  u^{-1}.
\end{equation}    
The problem of finding the element $s$ will be split in two separate problems
\begin{enumerate}
\item
find $\tilde{\xi}_{1} \in \mathcal{O}(\Lambda_{1, \mu})$, resp. $\tilde{\xi}_{2} \in  \mathcal{O}(\Lambda_{2, \mu})$, lifts of the elements $\xi_{1}$, resp.  $\xi_{2}$,
\item
apply Theorem \ref{liftinginvertibles}   
with $s_{1}$, resp. $s_{2}$,  replaced by $s_{1} + \tilde{\xi}_{1}$, resp. $s_{2} + \tilde{\xi}_{2}$, and the invertible element $u$.
\end{enumerate}  
1. The first lift is obtained in two steps: -i)  we find invertible lifts 
$\tilde{\xi}_{1} \in \Lambda_{1, \mu}$, resp. $\tilde{\xi}_{2} \in \Lambda_{2, \mu}$,
of the elements $\xi_{1}$, resp. $\xi_{2}$ , and we verify that  -ii) such lifts belong to $\mathbf{O}_{2n}(\Lambda_{i, \mu})$,  $i = 1, 2.$ 
\par
As the lifts of the elements $\xi_{i}$ belong to  $\mathbf{O}_{2n}(\Lambda_{i, \mu})$,  the
lifted elements will not change the corresponding $T$-completion classes. 
\subsection{Lifting of $\mathbb{O}(\Lambda^{'}_{\mu})$}  \label{liftingO}  
We illustrate the procedure on the element $\xi_{1}$; \;for $\xi_{2}$ we use the same procedure.
\par
The element $\xi_{1}$ has the form 
\begin{equation}  
\xi_{1} =   
\begin{pmatrix}
\alpha_{1}  &  0 \\
0  &   \alpha_{1}^{-1}
\end{pmatrix}
=
\alpha_{1}  \oplus \alpha_{1}^{-1} 
\in \mathbf{O}_{2n}(\Lambda^{'}_{\mu}).
\end{equation}  
\par
We use the factorisation (\ref{factorisation})   
to decompose of $\xi_{1}$
\begin{equation}  
\xi_{1} = 
\begin{pmatrix}
\alpha_{1} & 0 \\
0 & \alpha^{-1}_{1}
\end{pmatrix}
=
\begin{pmatrix}
1 & \alpha_{1} \\
0 & 1
\end{pmatrix}
\begin{pmatrix}
1 & 0 \\
- \alpha^{-1}_{1} & 1
\end{pmatrix}
\begin{pmatrix}
1 & \alpha_{1} \\
0 & 1
\end{pmatrix}
\begin{pmatrix}
0 & -1 \\
1 & 0
\end{pmatrix}
.
\end{equation}    
 
The homo-morphism $j_{1}$ being an epi-morphism, there exists
$\tilde{\alpha}_{1}$, resp. $\tilde{\beta}_{1} \in \Lambda_{1, \mu}$ such that
$j_{1}(\tilde{\alpha}_{1}) = \alpha_{1}$, resp.  $j_{1}(\tilde{\beta}_{1}) =  \alpha_{1}^{-1}.$ 
The lifted element is
\begin{equation}  \label{lifting}
\bar{\xi}_{1} 
=
\begin{pmatrix}
1 & \tilde{\alpha}_{1} \\
0 & 1
\end{pmatrix}
\begin{pmatrix}
1 & 0 \\
- \tilde{\beta}_{1} & 1
\end{pmatrix}
\begin{pmatrix}
1 & \tilde{\alpha}_{1} \\
0 & 1
\end{pmatrix}
\begin{pmatrix}
0 & -1 \\
1 & 0
\end{pmatrix}
,
\end{equation}    
\begin{equation*}  
j_{1}(\bar{\xi}_{1}) 
=
\begin{pmatrix}
\alpha_{1} & 0 \\
0 & \alpha_{1}^{-1}
\end{pmatrix}
= \xi_{1}.
\end{equation*}    
\par
We use formula (\ref{factorisation})   
and property 
(\ref{inverse}) 
of elementary matrices to find the inverse of $\bar{\xi}_{1}$
\begin{equation}  
\bar{\xi}_{1}^{-1} 
=
\begin{pmatrix}
0 & 1 \\
-1 & 0
\end{pmatrix}
\begin{pmatrix}
1 & - \tilde{\alpha}_{1} \\
0 & 1
\end{pmatrix}
\begin{pmatrix}
1 & 0 \\
 \tilde{\beta}_{1} & 1
\end{pmatrix}
\begin{pmatrix}
1 & - \tilde{\alpha}_{1} \\
0 & 1
\end{pmatrix}
.
\end{equation}    
We have
\begin{equation*}  
j_{1}(\bar{\xi}_{1}^{-1}) 
=
\begin{pmatrix}
\alpha_{1} & 0 \\
0 & \alpha_{1}^{-1}
\end{pmatrix}
^{-1}
\sim_{cl} \xi_{1}^{-1}
\end{equation*}    
because $j_{1}$ is a unital ring homo-morphism.
\par
The lifted element is
$$
\tilde{\xi}_{1} = \tilde{\xi}_{1} \oplus  \tilde{\xi}_{1}^{-1} \in 
\mathbb{M}_{4n}(\Lambda_{1, \mu}).
$$
\subsection{Back to the proof of Theorem \ref{liftinginvertibles}}  
 The elements $s_{1} + \tilde{\xi}_{1} \in \mathbb{GL}_{4n}(\Lambda_{1, \mu})$, 
                           $s_{2} + \tilde{\xi}_{2} \in \mathbb{GL}_{4n}(\Lambda_{2, \mu})$ and the invertible element $u$ satisfy the relation
\begin{equation}  
j_{1, \ast}  (s_{1} + \tilde{\xi}_{1}) = u \; (s_{2} + \tilde{\xi}_{2}) \; u^{-1}.
\end{equation}   
Theorem 42 follows from Theorem 43.   
\end{proof}  
\section{ $K_{1}(\mathcal{A})$  vs.  $T_{1}(\mathcal{A})$ }   \label{section21.13}
\par
The following identities are well known and used as building blocks of $K$-theory, see  \cite{Whitehead},  \cite{Milnor_Introduction},  \cite{Karoubi_1987}, \cite{Blackadar}, \cite{Rosenberg}, \cite{Rordam_Lansen_Lautsen}, \cite{Weibel}.
\par
We summarise some basic facts from the classical algebraic $K_{1}$-theory and compare them with $T_{1}$.
\begin{theorem}  
-i) Whitehead group $K_{1}(\mathcal{A})$    
is 
\begin{equation}   
K_{1}(\mathcal{A}) := \injlim_{n \in \mathbb{N}}  \mathbb{GL}_{n}(\mathcal{A})  /  [ \mathbb{GL}_{n}(\mathcal{A})  ,\mathbb{GL}_{n}(\mathcal{A}) ]
\end{equation}     
The group structure in $K_{1}(\mathcal{A})$ is given by matrix multiplication and direct sum addition
\begin{equation}  
[ A ] + [B] :=
[ \;A.B\;].
\end{equation}  
$T_{1}(\mathcal{A})$ is the set of Jordan canonical forms of matrices over $\mathcal{A}$ modulo
$\mathcal{O}(\mathcal{A})$. The sum in $T_{1}(\mathcal{A})$ is the direct sum.
\par                                              
-ii) Any commutator is stably isomorphic to a product of elements of the form 
$\mathcal{O}_{2n}(\mathcal{A})$. More specifically, for any $ A , B \in \mathbb{GL} (\mathcal{A}_{\mu})$
\begin{equation}    \ref{commutator} 
\begin{pmatrix}
ABA^{-1} B^{-1} & 0\\
0  &  1
\end{pmatrix}
=
\begin{pmatrix}
A & 0\\
0  &  A^{-1}
\end{pmatrix}
\begin{pmatrix}
B& 0\\
0  &  B^{-1}
\end{pmatrix}
\begin{pmatrix}
(BA)^{-1} & 0\\
0  &  BA
\end{pmatrix}.
\end{equation}     
-iii) If   $A \in \mathbb{GL}_{n}(\mathcal{A}_{\mu})$, then  (see 
 (\ref{factorisation}))  
\begin{equation}  
\mathcal{O}_{2n} (A) :=
\begin{pmatrix}
A & 0 \\
0  & A^{-1}
\end{pmatrix}
=
\begin{pmatrix}
1 & A \\
0  & 1
\end{pmatrix}
\begin{pmatrix}
1 & 0 \\
-A^{-1}  & 1
\end{pmatrix}
\begin{pmatrix}
1 & A \\
0  & 1
\end{pmatrix}
\begin{pmatrix}
0 & -1 \\
1  &  0
\end{pmatrix}
\end{equation}   
-iv) Any elementary matrix is a commutator  
\begin{equation}  
E_{ij} (A) = [ E_{ik}(A), E_{kj}(1)],   \hspace{0.2cm}  for \; any\; i, j, k \; distinct \; indices.
\end{equation}  
-iv) For any $A, B  \in \mathbb{GL}_{n}( \mathcal{A}_{\mu})$, $A+B$ is stably equivalent to $AB$ and $BA$ modulo (multiplicatively) elements of the form $\mathcal{O}_{2n}(\mathcal{A})$
\begin{equation}   
\begin{pmatrix}
A  &  0\\
0  &  B
\end{pmatrix}
=
\begin{pmatrix}
AB & 0\\
0  &  1
\end{pmatrix}  
\begin{pmatrix}
B^{-1} & 0\\
0  &  B
\end{pmatrix} 
=
\begin{pmatrix}
B^{-1} & 0\\
0  &  B
\end{pmatrix} 
\begin{pmatrix}
BA & 0\\
0  &  1
\end{pmatrix}  
. 
\end{equation}  
\par
-v) For any $ A, \; B \in \mathbb{GL}_{n} ( \mathcal{A}_{\mu}) $ one has the identity
\begin{equation}    
\begin{pmatrix}
ABA^{-1}  &  0 \\
       0         &  1
\end{pmatrix}
=
\begin{pmatrix}
A   &   0  \\
0   &   A^{-1}
\end{pmatrix}
\begin{pmatrix}
B  &   0 \\
0  &    1
\end{pmatrix}
\begin{pmatrix}
A^{-1}  &  0  \\
    0      &   A
\end{pmatrix} = 
\begin{pmatrix}
A & 0\\
0  &  A^{-1}
\end{pmatrix}
\begin{pmatrix}
B& 0\\
0  &  1
\end{pmatrix}
\begin{pmatrix}
A & 0\\
0  &  A^{-1}
\end{pmatrix}^{-1}
. 
\end{equation}    
\end{theorem}   

\begin{theorem} (see \cite{Whitehead}, \cite{Bass}, \cite{Swan}, \cite{Milnor_Introduction}, \cite{Karoubi_1987}, \cite{Rosenberg} )   
\par
-i)  $[\mathbb{GL}(\mathcal{A}),  \mathbb{GL}_{n}(\mathcal{A}) ]  =   \mathbb{E}_{n}(\mathcal{A})$
\par
-ii) $[A B A^{-1}B^{-1}]  = 0  \in K_{1} ( \mathcal{A})$   
\par
-iii.1) $[A ]+ [B]  = [ \;A B\; ] = [ \;BA\; ]  = [ B ]+ [ A] \in  K_{1} (\mathcal{A}) $. 
\par
-iii.2) $[A ]+ [B]  = [B] + [A] \in  T_{1}^{loc} (\mathcal{A})$.
Therefore $T_{1}^{loc}(\mathcal{A})$ is an Abelian group.
\par
-iv) $[\mathcal{O}_{2n} (\mathcal{A})]  =  [1_{n} ]   =  0 \; in \; K_{1}(\mathcal{A})$ and
$T_{1}^{loc}(\mathcal{A})$
\par
-v) $ - [A] = [A^{-1}]   \;in\;  K_{1}(\mathcal{A})\; and \;T_{1}(\mathcal{A})^{loc}$.
\par
-vi) 
$[A B A^{-1}]  = [ B ]   \in K_{1} ( \mathcal{A})$.   
\end{theorem}   

\begin{proof}  
-i) Relation (\ref{commutator}) 
says that any commutator is a product of matrices of type 
$\mathcal{O}_{n}(\mathcal{A})$. Formula (\ref{factorisation}) 
says that any matrix of type  $\mathcal{O}_{n}(\mathcal{A})$ is a product of elementary matrices and a scalar matrix
 \ref{factorisation}. This proves -i).
\par
-ii) Follows from the definition of $K_{1}(\mathcal{A})$.
\par
-iii.1) and -iii.2).  See definitions.
\par
-iv)  For any invertible element $x_{0}$  one has 
\begin{equation}
  (x_{0} +  (  \xi + \xi^{-1} )    ) \sim_{
  \mathcal{O}
  }  (x_{0}  + (\xi + \xi^{-1})) +  (\xi + \xi^{-1})
\end{equation}
along with the fact that  $T_{1}$   is a group.
\par
-v)  Relation for  $K_{1}$  follows from $A . A^{-1} = 1$, which is the zero 
element in  $K_{1}$. 
\par
For $T_{1}$ the relation follows from  -iv).
\par
-vi) $(A B A^{-1}) . B^{-1}$ is a commutator; it is the zero element in $K_{1}$. On the other side $[B^{-1}] \;=\; [B] \in K_{1}$.
 \end{proof}    
 \par
The next remark explains the specific parts of the construction of $T_{1}^{loc}$ which make our construction different from the classical one.
\begin{remark} 
\par
-i) In our construction the factorisation of the elements of $\mathbb{GL}(\mathcal{A})$ through the commutator sub-group  (or the sub-group generated by elementary matrices) is avoided because  \emph{the number of products needed to generate these sub-groups might be un-bounded}. Having in mind the algebra of integral operators, or pseudo-differential operators,  one realises that any product increases the size of the support;  an un-bounded number of products does not allow a support control.
\par
For this reason, in the whole chapter we don't use more than \emph{three} products of elements in the algebra. Products are replaced by sums and direct sums.  The increase of the size of matrices replaces in our construction the need to perform multiple products. In our definition of $T_{1}^{loc}(\mathcal{A})$, where products could not be avoided, the corresponding increase in the size of the supports  is absorbed by the projective limit. 
\par
-ii) Our construction uses the elements $\mathcal{O}(\mathcal{A})$ \emph{additively} and not \emph{multiplicatively}. In the classical construction of $K_{1}(\mathcal{A})$,  the elements of $\mathcal{O}(\mathcal{A})$ are used multiplicatively. Unfortunately, to generate the commutator sub-group it is necessary to perform an \emph{un-bounded number of products};  in our construction an un-bounded number of products is not allowed. 
\par
-iii) The construction of $T^{loc}_{1}(\mathcal{A})$ uses the factorisation of $\mathbb{GL}(\mathcal{A})$ through the smaller (than the commutator) sub-group of inner auto-morphisms. The class of an invertible element $u$ modulo inner auto-morphisms, which is nothing but the abstract Jordan form $J(u)$ of 
$u$,  contains more information than the class of the invertible element modulo the commutator sub-group.
Both, in the classical $K_{1}(\mathcal{A})$ and $T_{1}(\mathcal{A})$ the elements of 
$\mathcal{O}_{n}(\mathcal{A})$ represent the zero element. 
\par
If $\mathcal{A}$ were the algebra of complex matrices and $u \in \mathbb{M}(\mathbb{C})$, then
$J(u)$ could be identified with the Jordan canonical form of the matrix $u$. The Jordan canonical form of the matrix $u \oplus u^{-1}$ is precisely $J(u) \cup J(u^{-1})$ modulo permutations of the Jordan blocks. 
It is clear that $ u \oplus u^{-1} $ could never be conjugated to the the identity element unless $u = 1_{n}$. In general,  $u \oplus u^{-1}$ could not be conjugated to $1_{n} \oplus 1_{n}$ unless $u \sim_{s} 1_{n}$. 
 \par
 For the computation of the groups $T_{i}^{loc}(\mathbb{C})$, where $\mathbb{C}$ is the algebra of complex numbers endowed with the trivial filtration, see \cite{Teleman_arXiv_V} and \S 22.  
\par
The elements $\mathcal{O}(\mathcal{A})$ represent the zero element in $T_{1}$. 
\par
 For the definition of  $T_{1}^{loc}(\mathcal{A})$ we find it natural to consider the 
quotient space of $\mathbb{GL}(\mathcal{A})$ modulo the equivalence relation
$\sim_{sl}$. This factorisation uses the \emph{additive sub-group} 
$\mathcal{O}(\mathcal{A})$. This factorisation decrees that the Jordan canonical forms of the elements $u$ and  $u^{-1}$ are opposite one to each other in $T_{1}^{loc}$.
The \emph{additive} group generated by
elements $\mathcal{O}(\mathcal{A})$ is contained in the commutator sub-group; this property insures the fact that there exists a natural epi-morphism from  $T_{1}^{loc}(\mathcal{A})$ to  $K_{1}(\mathcal{A})$.
\par
-iv) The factorisation through the sub-group $\mathcal{O}(\mathcal{A})$  does not appear to kill much information. A partial argument in support of this is the fact that for any $u, v \in 
\mathbb{GL}_{n}(\mathcal{A})$, 
$(u \oplus u^{-1}) \sim_{l} (v \oplus v^{-1})$ if and only if $(u \oplus 1_{n}) \sim_{l} (v \oplus 1_{n})$. 
\par
-v) Additionally, the projective limit (Alexander-Spanier type construction, made possible by the filtration $\mathcal{A}_{\mu}$ of the algebra), makes the algebraic $T^{loc}_{i}$-theory richer than the classical $K_{i}$-theory, $i = 0,  1$.  
\end{remark}  

\begin{theorem}    
-i) There is a canonical epi-morphism
\begin{gather}   
\Pi:  T^{loc}_{1} (\mathcal{A})  \longrightarrow   K_{1} (\mathcal{A}) \\
Ker \; \Pi =   [\; \mathbf{GL}(\mathcal{A}),  \mathbf{GL}(\mathcal{A}) \;] \; /  \;Inner(\mathcal{A}).
\end{gather}   
\end{theorem}  
\section{Connecting homo-morphism   
$\partial:  T_{1}^{loc} (\Lambda^{'}) \longrightarrow T_{0}^{loc}(\Lambda)\otimes \mathbb{Z}[\frac{1}{2}] 
$}                                                                     
In this section we assume that the diagram \ref{MayerVietoris} satisfies 
Hypotheses 1, 2. 3.
In this section we define the \emph{connecting homomorphism}
$\partial: T_{1}^{loc} \longrightarrow T_{0}^{loc} \otimes \mathbb{Z}[\frac{1}{2}]$. 
\par
The reader will notice that  the construction of $\partial$ involves from the very beginning idempotents.
\begin{definition}    \label{partial} 
Let $[u] \in T_{1} (\Lambda^{'}_{\mu})$. Recall that the elements of 
$T_{1}(\Lambda^{'}_{\mu})$ are equivalence classes of invertible matrices modulo 
$\sim_{\mathcal{O}_{\mathcal{A}_{\mu}}}$; we may assume that 
$u \in \mathbb{GL}_{n}(\Lambda^{'}_{\mu})$. 
\par
Define the \emph{connecting homomophism}
$\partial:  T_{1}(\Lambda^{'}_{\mu}) \longrightarrow   T_{0}(\mathcal{A}_{\mu}) \otimes \mathbb{Z}[\frac{1}{2}]$
\begin{equation}  
\partial[u] =  [p(1_{n}, 1_{n}, \tilde{U}(u)] - [\Lambda^{n}(\mathcal{A}_{\mu})],
\end{equation}  
where $\tilde{U}_{1}(u)$ is obtained through the decomposition of $u \oplus u^{-1}$ as in 
(\ref{factorisation}), and lifted in $\Lambda_{1, \mu}$, as in (\ref{lifting}).        
\end{definition} 
\par  
\begin{proposition}   
Suppose the homo-morphisms $j_{1}, j_{2}$ are epi-morphisms.
\par
Then the connecting homomorphism is well defined.
\end{proposition}   
We have to show that $\partial$ is compatible with the equivalence relation 
$\sim_{\mathbb{O}_{\mathcal{A}_{\mu}}}$.  For,  suppose that 
$u_{1} + \xi_{1} = u_{0} + \xi_{0}$,
where $u_{0}, u_{1} \in \mathbb{GL}_{n}(\mathcal{A}_{\mu})$
and  $\xi_{0}, \xi_{1} \in \mathbb{O}_{n}(\mathcal{A}_{\mu})$. 
We have to prove that 
$\partial \mathbb{O}(u_{1}) = \partial \mathbb{O}(u_{2}) = 
trivial \; idempotent$.
\par
We proved in \S \ref{liftingO} that any invertible element $\xi$ belonging to 
$\mathbb{O}_{n} (\Lambda^{'}_{\mu})$ lifts to an invertible elements 
$\tilde{\xi}_{1}$, resp. $\tilde{\xi}_{2}$, belonging to 
$\Lambda_{\mu,1}$, resp.  $\Lambda_{\mu,2}$. In our specific case, the element
$1_{n}(\Lambda_{1, \mu})$, resp.  $1_{n}(\Lambda_{2, \mu})$, is conjugated through the
lifted elements $\tilde{\xi}_{i}$, $i = 1, 2$.
\par
The conjugation inside $\Lambda_{\mu, i}$ preserves the $T_{1}(\Lambda_{\mu, i})$ classes. 
 Recall that the idempotents of 
$\Lambda_{\mu}$ consist of pairs of idempotents $(p_{1}, p_{2}) \in
\Lambda_{\mu, 1} \oplus \Lambda_{\mu, 2}$ such that
$i_{1} p_{1} = i_{2} p_{2}$. The idempotents 
$$
p_{1} = \tilde{\xi}_{1} 1_{n} \tilde{\xi}_{1}^{-1},  \;\;\;
p_{2} = \tilde{\xi}_{2} 1_{n} \tilde{\xi}_{2}^{-1}
$$
satisfy this condition. The proposition is proven.
\section{Six terms exact sequence.}    
\begin{theorem} \ref{exactsequence}   
The following sequences are exact
\par
{\bf -i)}   
\begin{gather}   
T_{0}^{loc} (\Lambda)   
\overset{ (i_{1},i_{2}) }{\longrightarrow} 
T_{0}^{loc} (\Lambda_{1})  \oplus T_{0}^{loc} (\Lambda_{2}) 
\overset{ j_{1 \ast} - j_{2 \ast} }{\longrightarrow}   
T_{0}^{loc}(\Lambda^{'})
\end{gather}  
\par
{\bf -ii)}
\begin{gather}    
T_{1}^{loc} (\Lambda)  
\overset{(i_{1 \ast}, i_{2 \ast})}{\longrightarrow}
T_{1}^{loc} (\Lambda_{1})   \oplus T_{1}^{loc} (\Lambda_{2}) 
\overset{ j_{1 \ast} - j_{2 \ast} }{\longrightarrow} 
T_{1}^{loc} (\Lambda^{'})  
\end{gather}   
\par
{\bf -iii)}
\begin{gather}  
T_{1}^{loc} (\Lambda_{1})  \oplus T_{1}^{loc} (\Lambda_{2})  
\overset{ j_{1 \ast} - j_{2 \ast} }{\longrightarrow} 
T_{1}^{loc} (\Lambda')   \overset{\partial}{\longrightarrow} \\
T_{0}^{loc} (\Lambda) \otimes \mathbb{Z}[\frac{1}{2}]
  \overset{ (i_{1 \ast}, i_{2 \ast}) }{\longrightarrow} 
 (T_{0}^{loc} (\Lambda_{1})  \otimes \mathbb{Z}[\frac{1}{2}]) \oplus 
 (T_{0}^{loc} (\Lambda_{2})  \otimes \mathbb{Z}[\frac{1}{2}]). 
\end{gather} 
Therefore, the following six terms sequence is exact
\begin{gather}    
T_{1}^{loc} (\Lambda)  
\overset{ (i_{1 \ast}, i_{2 \ast}) }{\longrightarrow}
T_{1}^{loc} (\Lambda_{1}) \oplus T_{1}^{loc} (\Lambda_{2}) 
\overset{ j_{1 \ast} - j_{2 \ast} }{\longrightarrow} 
T_{1}^{loc} (\Lambda^{'})  
\overset {\partial}{\longrightarrow} \\
T_{0}^{loc} (\Lambda)  \otimes \mathbb{Z}[\frac{1}{2}] 
\overset{(i_{1 \ast}, i_{2 \ast})}{\longrightarrow}
T_{0}^{loc} (\Lambda_{1})  \otimes \mathbb{Z}[\frac{1}{2}]  \oplus 
T_{0}^{loc} (\Lambda_{2})  \otimes \mathbb{Z}[\frac{1}{2}] 
\overset{ j_{1 \ast} - j_{2 \ast}}{\longrightarrow} 
T_{0}^{loc} (\Lambda^{'}) \otimes \mathbb{Z}[\frac{1}{2}] 
\end{gather}   
\end{theorem} 
\begin{proof}   
\par

{\bf -i)} 
\par
{\bf i.1.)} \; $Im (i_{1 \ast}, i_{2 \ast}) \subset Ker (j_{1 \ast} - j_{2 \ast})$.
\par
Easy to verify.
\par
{\bf i.2.)} \; $Ker (j_{1 \ast} - j_{1 \ast}) \subset  Im  (i_{1 \ast}, i_{2 \ast})$
\par
We are going to verify it.
\begin{lemma}  
Let $p \in \mathbf{I}demp_{n} (\mathcal{A}_{\mu})$.
\par
Then  
\begin{equation}  
p \oplus (1-p) \sim_{ls} 1_{n}.
\end{equation}  
\end{lemma}  
\begin{proof}   
 The proof is based on the following identity
\begin{gather}   
\begin{pmatrix}
p & 0 \\
0 & 1-p
\end{pmatrix}
=
\begin{pmatrix}
p & 1-p \\
1-p & p 
\end{pmatrix}
\begin{pmatrix}
1_{n} & 0 \\
0 & 0
\end{pmatrix}
\begin{pmatrix}
p & 1-p \\
1-p & p 
\end{pmatrix}
\end{gather}   
along with the observation that

\begin{gather}   
\begin{pmatrix}
p & 1-p \\
1-p & p
\end{pmatrix}
^{2}
=
\begin{pmatrix}
p^{2} + (1-p)^{2} & p(1-p) + (1-p)p\\
(1-p) p + p(1-p) & p^{2} + (1-p)^{2} 
\end{pmatrix}
=
\begin{pmatrix}
1_{n} & 0 \\
0 & 1_{n}
\end{pmatrix}
\end{gather}    
which shows that
\begin{gather}   
\begin{pmatrix}
p & 1-p \\
1-p & p
\end{pmatrix}
^{-1} =
\begin{pmatrix}
p & 1-p \\
1-p & p
\end{pmatrix},
\end{gather}    
which entitles us to say that the RHS of (107) is an inner auto-morphism.
\end{proof}  

For any idempotent $p$ of size $n$ we convine to write $\bar{p} := 1_{n} - p$.  
\par
Let $([p_{1}] - [p_{2}] \; , \;  [q_{1}] - [q_{2}]) \in T_{0}^{loc} (\Lambda_{1, \mu}) \oplus T_{0}^{loc} (\Lambda_{2, \mu}) $
be such that 
$$
0 = j_{1, \ast}([p_{1}] - [p_{2}]) - j_{2, \ast} ([q_{1}] - [q_{2}]), 
$$
where $p_{1}, p_{2}, q_{1}, q_{2}$ are idempotents. The pair of $T$-theory classes may be re-written
$$
0 = j_{1, \ast}([p_{1} + \bar{p}_{2}] - [p_{2}+ \bar{p}_{2}]) - j_{2, \ast} ([q_{1}+ \bar{q}_{2}] - [q_{2}+ \bar{q}_{2}]) 
$$
or
$$
0 = j_{1, \ast}([p_{1} + \bar{p}_{2}] - [1_{r}]) - j_{2, \ast} ([q_{1}+ \bar{q}_{2}] - [1_{s}]). 
$$
By adding all sides a trivial idempotent of sufficiently large size, we may assume that $r = s$ is large. This relation may be re-written
$$
0 = j_{1, \ast}([p_{1} + \bar{p}_{2}] - j_{2, \ast} ([q_{1}+ \bar{q}_{2}] ). 
$$
This means that there exists an idempotent $\xi \in \mathbf{I}demp_{N}(\Lambda^{'}_{\mu})$  such that the idempotents
$$
( j_{1, \ast}([p_{1} + \bar{p}_{2}] + \xi) - ( j_{2, \ast} ([q_{1}+ \bar{q}_{2}] + \xi) 
$$
are isomorphic. We add further the idempotent $\bar{\xi} := 1_{N} - \xi$ to get isomorphic idempotents

$$
( j_{1, \ast}([p_{1} + \bar{p}_{2} ] + \xi + \bar{\xi}),   \;\;  ( j_{2, \ast} ([q_{1}+ \bar{q}_{2}] + \xi + \bar{\xi}) \in 
\mathbb{I}demp_{\bar{N}} (\Lambda^{'}_{\mu})
$$
( $\bar{N}$ being the size of these idempotents)
or
$$
( j_{1, \ast}([p_{1} + \bar{p}_{2} ] + 1_{N}), \;\;  ( j_{2, \ast} ([q_{1}+ \bar{q}_{2}] + 1_{N}). 
$$
This means there exists $u \in \mathbf{GL}_{\bar{N}}(\Lambda^{'})$ which conjugates these two idempotents.
\par
Theorem \ref{theorem21.2}   
says that there exits the idempotent $p \in \mathbb{I}demp_{2 \bar{N}}(\Lambda_{\mu})$ such that
$$
i_{1 \ast} p = (p_{1} + \bar{p}_{2} + 1_{N}) \oplus 1_{\bar{N}}
$$
and 
$$
i_{2 \ast} p = U \; ( (q_{1} + \bar{q}_{2} + 1_{N}) \oplus 1_{\bar{N}} ) \; U^{-1}
$$
where 
$$
U \in \mathbb{GL}_{2 \bar{N}}(\Lambda^{'}_{\mu}).
$$
This means the image of the $T$-theory class of $p - 1_{\bar{N} } \in \Lambda_{\mu}$ through the pair of homomorphisms $(i_{1, \ast}, i_{2, \ast})$ is
$$
(i_{1, \ast}, i_{2, \ast}) [p - 1_{\bar{N}}] =   
$$
$$
= (\;  (p_{1} + \bar{p}_{2} + 1_{N}) \oplus 1_{\bar{N}} - 1_{\bar{N}}, \;  
U \; ( (q_{1} + \bar{q}_{2} + 1_{N}) \oplus 1_{\bar{N}} ) \; U^{-1} - 1_{\bar{N}}
\; ) = 
$$

$$
= (\;  (p_{1} + \bar{p}_{2} + 1_{N}) \oplus 1_{\bar{N}} - 1_{\bar{N}},  \;
U \; ( (q_{1} + \bar{q}_{2} + 1_{N}) \oplus 1_{\bar{N}}  - 1_{\bar{N}}  ) \;  \; U^{-1} 
\;  = 
$$
$$
= ([p_{1} - p_{2}], [q_{1} - q_{2}],
$$
which completes the proof of the part {\bf i)}.
\par

{\bf -ii)}  
\par
{\bf -ii.1.)}   $Im (i_{1}, i_{2})_{\ast} \subset Ker (j_{1,\ast} - j_{1,\ast})$. 
\par
If $u = (u_{1}, u_{2}) \in \mathbb{GL}(\Lambda_{\mu})$  belongs to $Im (i_{1}, i_{2})_{\ast}$, then according to the definition of $\Lambda_{\mu}$,
$u$ may be chosen so that $j_{1} (u_{1}) = j_{2} (u_{2})$ and hence  $[u] \in Ker (j_{1, \ast} - j_{2, \ast})$. 
\par
{\bf -ii.2.)}  $Ker (j_{1,\ast} - j_{2,\ast}) \subset Im (i_{1}, i_{2})_{\ast}$. 
\par
Let $u =  (u_{1}, u_{2}) \in Ker (j_{1,\ast} - j_{1,\ast})$.
This means there exist two elements $\xi_{1}, \xi_{2} \in \mathbb{O}(\Lambda^{'}_{\mu})$ so that
$j_{1}(u_{1}) + \xi_{1} =  j_{2}(u_{2}) + \xi_{2}$. The elements $j_{1}(u_{1})$, $j_{2}(u_{2})$  lift
 to  $u_{1} \in \Lambda_{1, \mu}$, resp. $u_{2} \in \Lambda_{2, \mu}$. On the other side, see  \S \ref{liftingO},   
 the elements 
 $\xi_{1}$, resp. $\xi_{2}$, lift to elements $\tilde{\xi}_{1} \in \mathbb{O}(\Lambda_{1, \mu})$, resp. $\tilde{\xi}_{2} \in  \mathbb{O}(\Lambda_{2, \mu})$.  This means
 \begin{equation*}
 u_{1} + \tilde{\xi}_{1} \in \Lambda_{1, \mu}  \;\; and\;\;  u_{1} + \tilde{\xi}_{2} \in \Lambda_{2, \mu}
 \end{equation*}
 are so that 
 $$ 
j_{1}( u_{1} + \tilde{\xi}_{1}) =  j_{2}( u_{2} + \tilde{\xi}_{2}). 
 $$
 Therefore,
 $$[ ( u_{1} + \tilde{\xi}_{1}), \; ( u_{2} + \tilde{\xi}_{2})] = [ ( u_{1}, \; u_{2} )] \in T_{1}(\Lambda_{1\mu}) 
 \oplus T_{1}(\Lambda_{2\mu}),$$
 which proves the statement. 
\par
{\bf -iii)} 
\par
{\bf -iii.1.)} $Im ( j_{1 \ast} - j_{2 \ast})   \subset Ker \partial$. 
\par
It is sufficient to prove  $\partial \circ  j_{1 \ast} =0$. Consider  $u \in \Lambda_{1, \mu}$. 
We have to compute  $\partial \circ j_{1}(u)$. The result is independent of the lift of $(j_{1}(u)) \otimes (j_{1}(u))^{-1}$.  
We may use the lift $u \otimes u^{-1}$. We get then
\begin{equation} 
\partial \circ j_{1}(u) = (u \otimes u^{-1}) (1_{n} \oplus 1_{n}) (u \otimes u^{-1})^{-1} =  1_{n} \oplus 1_{n}.
\end{equation}  
 {\bf -iii.2.)} $Ker \partial \subset Im ( j_{1 \ast} - j_{2 \ast}) $. 
 Let $u \in \Lambda^{'}_{\mu}$ be such that 
 $$\partial u =  [( 1_{n}, 1_{n}, U(u) )] - [\Lambda ( \mathcal{A}_{\mu}^{n})] = 0.$$
 We have to prove that there exist  $v_{1} \in \Lambda_{1, \mu}$, $v_{2} \in \Lambda_{2, \mu}$ such that  $][u = [j_{1, \ast} v_{1} - j_{2, \ast} v_{2}]$.
 \par
 The hypothesis says
 $$
 [( 1_{n}, 1_{n}, U (u) )] = [\Lambda_{\mu}^{n}].
 $$
 By adding an idempotent $q$  and its complementary to both sides of 
 this equation, we get
 $$
  (1_{m+n}, 1_{m+n}, U (u \oplus 1_{m}) ) \sim_{c} \Lambda_{\mu}^{m+n}.
 $$
 We may transfer the conjugation onto the second term of this equation to get
$$
 \partial  (u) =  \Lambda_{\mu}^{m+n},
$$
which proves the statement.
\par 
{\bf -iii.3.)} 
$Im \; \partial  \subset Ker \; (i_{1},  i_{2})_{\ast}$.  
\par
Let $u \in T_{1}(\Lambda^{'}_{\mu})$  and  $p = \partial u \in T_{0}(\Lambda_{\mu})$.  
We have to prove that  
$(i_{1}, i_{2})_{\ast}(p) = 0$. 
\par
We have
\begin{equation}
[i_{1} \circ \partial u] = [i_{1} \;(1_{n} \oplus 0_{n},   \;u (1_{n} \oplus 0_{n}) {u}^{-1}) 
- \Lambda^{n}  (\mathcal{A}_{\mu})]
= 
\end{equation}
\begin{equation*}
[1_{n} \oplus 0_{n}] -  [\Lambda^{n}  (\mathcal{A}_{\mu})] = 0  
\end{equation*}
On the other side
\begin{equation}
[i_{2} \circ \partial u] = [i_{2} \; (1_{n} \oplus 0_{n},   \; u (1_{n} \oplus 0_{n}) u^{-1})] 
- [\Lambda^{n}  (\mathcal{A}_{\mu}]
= 
\end{equation}
\begin{equation*}
 [u (1_{n} \oplus 0_{n}) u^{-1})]  -  [\Lambda^{n}  (\mathcal{A}_{\mu})] = 0.
\end{equation*}
\par
{\bf -iii.4.)} 
$ Ker (i_{1},  i_{2})_{\ast} \subset Im (\partial \otimes \mathbb{Z}[\frac{1}{2}])$. 
Let $p \in T_{0}(\Lambda_{\mu})$  and  $(i_{1}, i_{2})_{\ast}(p) = 0$. 
We have to prove that there exists $u \in T_{1}(\Lambda^{'}_{\mu})$ such that
$p = \partial u \in T_{0}(\Lambda_{\mu})$.
\par
That is, we look for an invertible matrix $u \in \mathbb{M}_{n}(\Lambda^{'}_{\mu})$ such that 
\begin{gather*}
0 = (i_{1}, i_{2})_{\ast}([p]) =  (i_{1}, i_{2})_{\ast} 
[(1_{n} \oplus 0_{n},  u (1_{n} \oplus 0_{n})  u^{-1})] - [\Lambda_{\mu}^{n}] = \\
( [1_{n} \oplus 0_{n}],  [u (1_{n} \oplus 0_{n})  u^{-1}]) - [\Lambda_{\mu}^{n}].
\end{gather*}
We stabilise the idempotents $1_{n} \oplus 0_{n}$, $\Lambda_{\mu}^{n}$. 
The conjugation transforms trivial idempotents in trivial idempotents.
The stabilised idempotents $1_{n+m} \oplus 0_{n+m}$, $\Lambda_{\mu}^{n+m}$ still satisfy the condition
\begin{equation}
p_{1, \ast} (1_{n+m} \oplus 0_{n+m}) = 0,  \;\; p_{2, \ast} (\Lambda_{\mu})^{n+m} = 0.
\end{equation}
We choose an isomorphism $\tilde{u}$ between the trivial idempotents
$1_{n+m} \oplus 0_{n+m}$, $\;\Lambda_{\mu}^{n+m}$.
We reduce these idempotents, modulo the ideal $J$; we obtain idempotents in 
$\Lambda^{'}(\mathcal{A}_{\mu})$.  Let $u$ be the restriction of $\tilde{u}$ modulo $J$.
The element $u$ is the element we are looking for.
\par
If we started with the element $p \otimes 1/2^{m} \in T_{0}(\Lambda_{\mu}) \otimes \mathbb{Z}[\frac{1}{2}]$, the corresponding element would be $u\otimes 1/2^{m}$.
\end{proof}
\section{Relative $T$- groups: $T_{i}(\mathcal{A}_{\mu}, J)$.}  
\begin{definition}  
Let $\mathcal{A}$ be a localised ring and $\mathcal{J}$ be a localised bi-lateral ideal in $\mathcal{A}$.  
Here we make reference to the Mayer - Vietoris Diagram \ref{MayerVietoris};  
we take
\begin{equation}
\Lambda_{\mu} = \mathcal{J}_{\mu}, \;\;\; \Lambda_{1, \mu} = \Lambda_{2, \mu} = \mathcal{A}_{\mu} 
\end{equation}
\begin{equation}
\Lambda_{\mu} = 
\{ 
(\lambda_{1}, \lambda_{2})\;|\; \lambda_{1} \in \Lambda_{1, \mu}, \;
\lambda_{2} \in \Lambda_{2, \mu}, \; \lambda_{1} - \lambda_{2} \in \mathcal{J}_{\mu}.
\}
\end{equation}
and
\begin{equation}
\Lambda^{'}_{\mu} = \mathcal{A}_{\mu} / \mathcal{J}_{\mu} 
\end{equation}
The above structure will be called $(\mathcal{A}, \mathcal{J})$   \emph{localised ideal}.
\par
Define, for $i = 1, 2$
\begin{equation}
j_{i}:  \Lambda_{\mu} \longrightarrow \mathcal{A} /  \mathcal{J}, \;\;\;
j_{i}\; (\lambda_{1}, \lambda_{2}) := \lambda_{i}\; mod. \mathcal{J}_{\mu}.
\end{equation}
\end{definition}   
\begin{definition}  \label{definition21.17} 
\begin{gather}   
T_{0}^{loc}(\mathcal{A}, \mathcal{J}) := Ker \; (\;  
T_{0}^{loc} (\Lambda_{\mu})
\overset{j_{2, \ast}}\longrightarrow 
T_{0}^{loc}(\mathcal{A}_{\mu} / \mathcal{J}_{\mu} )\;) \\
T_{1}^{loc}(\mathcal{A}, \mathcal{J}) := Ker \;( \; 
T_{1}^{loc} (\Lambda_{\mu}) 
\overset{j_{2}, \ast}\longrightarrow 
T_{1}^{loc}(\mathcal{A}_{\mu} / \mathcal{J}_{\mu} )\;) 
\end{gather}    
\end{definition}   
\begin{theorem} \label{theorem21.9}   
\par
-i) One has the exact sequence
\begin{equation}   
0 \longrightarrow
\mathcal{J}_{\mu}
 \overset{\iota}{\longrightarrow}
\mathcal{A}_{\mu}
 \overset{\pi}{\longrightarrow}
\mathcal{A}_{\mu} / \mathcal{J}_{\mu}
\longrightarrow
0,
\end{equation}   
where $\iota$ is the inclusion and $\pi$  is the canonical projection.
\par
-ii) The exact sequence \ref{exactsequence} 
becomes
\begin{gather}   \label{secondexactsequence}   
T_{1}^{loc}(\mathcal{A}, \mathcal{J}) \overset{i_{\ast}}{\longrightarrow}
T_{1}^{loc}(\mathcal{A}) \overset{\pi_{\ast}}{\longrightarrow}
T_{1}^{loc}(\mathcal{A}/\mathcal{J}) \overset{\partial}{\longrightarrow} \\
T_{0}^{loc}(\mathcal{A}, \mathcal{J}) 
\otimes \mathbb{Z}[\frac{1}{2}]
\overset{i_{\ast}}{\longrightarrow}
T_{0}^{loc}(\mathcal{A}) 
\otimes \mathbb{Z}[\frac{1}{2}]
\overset{\pi_{\ast}}{\longrightarrow}
T_{0}^{loc}(\mathcal{A}/\mathcal{J})
\otimes \mathbb{Z}[\frac{1}{2}].
\end{gather}
\end{theorem}  
\begin{proof}
We leave the check to the reader.
\end{proof}
\section{Connecting homo-morphism - second form}  
\label{connectinghomomorphismsecondform}  
In this section we define Mayer - Vietoris diagramms associated with elliptic operators between two different Hilbert spaces and we exibit the corrresponding connecting homomorphism.  
\par
To motivate the next definition, let 
$
D: H_{0} \longrightarrow H_{1}
$
be an (integral or pseudo-differential) elliptic operator.  
Let $\sigma(D) := D \; mod. \; compact\;operators$ be its \emph{symbol}.
\par
We consider the ring
\begin{gather}
\Lambda_{1} = \Lambda_{2} := Hom_{vect} (H_{0} \oplus H_{1}, \; H_{0} \oplus H_{1}) \\
J := Compact \; operators \;  on \; H_{0} \oplus H_{1}
\end{gather}
\begin{definition} \label{definition21.17}   
Consider the following structure. 
\par
Let $\Lambda_{1} = \Lambda_{2}$ be a ring. Let $J \subset
\Lambda_{i}$, $i= 1, 2$,  be a bi-lateral ideal.
\par
Introduce the rings
\begin{gather}
\Lambda := \{      
(\lambda_{1},  \lambda_{2}) \;|\; \lambda_{1} \in \Lambda_{1}, \; \lambda_{2} \in \Lambda_{2}, 
\; \lambda_{1} - \lambda_{2}  \in J 
\} \\
\Lambda^{'} := \Lambda_{1} / J = \Lambda_{2} / J. 
\end{gather}
The ring structure in $\Lambda$ is
$$
(\lambda_{1},  \;\lambda_{2}) + (\lambda_{1}^{'},  \;\lambda_{2}^{'}) :=  
(\lambda_{1} + \lambda_{1}^{'},  \; \lambda_{2} + \lambda_{2}^{'})
$$
$$
(\lambda_{1},  \; \lambda_{2}) \;. \;(\lambda_{1}^{'},  \; \lambda_{2}^{'}) :=  
(\lambda_{1} . \lambda_{1}^{'},  \; \lambda_{2} . \lambda_{2}^{'})
$$
Define $i_{i}$, $\; i = 1, 2$,  
\begin{equation}
i_{i}: \Lambda  \longrightarrow \Lambda_{1} \oplus  \Lambda_{2},  \;\;\;
i_{i} \;(\lambda_{1}, \;\lambda_{2}) := \lambda_{i}.
\end{equation}
and
\begin{equation}
j_{i} : \Lambda_{1} \oplus \Lambda_{2} \longrightarrow \Lambda^{'} , \;\;\; 
j_{i}  \; (\lambda_{1},  \lambda_{2})  :=   \lambda_{i} / J.
\end{equation} 
\end{definition}
\begin{proposition}  \label{proposition21.9}  
$\{ \Lambda, \Lambda_{1}, \Lambda_{2}, \Lambda^{'}, i_{1},
i_{2}, j_{1}, j_{2} \} $ is a Mayer-Vietoris diagramm,  see 
\S \ref{MayerVietoris}.  
\par
If $\Lambda_{i}, J$  are \emph{localised} rings, then this is a localised Mayer - Vietoris diagramm.
\par
To this localised Mayer-Vietoris diagram  there corresponds a six term exact sequence, see Theorem \ref{exactsequence}.
\end{proposition}    
Next, we are going to exibit the corresponding connecting homo-morphism
 $\partial: T^{loc}_{1}(\Lambda^{'}) \longrightarrow T^{loc}_{0}(\Lambda)$. 
\begin{definition}   
Let $[u] \in T_{1}^{loc}(\Lambda^{'})$, where 
$u \in \mathbb{GL}_{n}(\Lambda^{'}_{\mu})$ for some $n$ and $\mu$.
\par
Let $A,  \; resp. \;B,  \in  \mathbb{M}_{n} (\Lambda_{1, \mu})$  be liftings of $u$, resp. $u^{-1}$, in 
$\mathbf{M}_{n} (\Lambda_{1, \mu})$. Such liftings exist because the canonical mapping $j_{1}$ is surjective. 
\par
With $A$ and $B$ one associates 
$$S_{0}= 1 -  BA \in \mathbf{M}_{n} (\Lambda_{1, \mu -1 }), \;\;\;
S_{1}= 1 -  AB \in \mathbf{M}_{n} (\Lambda_{1, \mu -1 }).$$
The matrices $S_{0}$, $S_{1}$ satisfy 
\begin{equation}  
j_{1} (S_{0}) = j_{1} (S_{1}) = 0.
\end{equation}  
With these matrices one associates the invertible matrix    
(ref.  \cite{Connes_Moscovici})
\begin{gather}   
L =
\begin{pmatrix}
S_{0}  &  - (1 + S_{0}) B  \\
A  &  S_{1}
\end{pmatrix}
\in \mathbb{GL}_{2n}(\Lambda_{1, \;\mu - 2});
\end{gather}   
the inverse of the matrix $L$ is 
\begin{gather}  
L^{-1} =
\begin{pmatrix}
S_{0}  &   (1 + S_{0}) B  \\
- A  &  S_{1}
\end{pmatrix}
\in \mathbb{GL}_{2n}(\Lambda_{1, \;\mu - 2}).
\end{gather}    
Let $e_{1},  e_{2} $ be the idempotents
\begin{gather}  
e_{1} =
\begin{pmatrix}
1 &   0  \\
0  &  0
\end{pmatrix}
\in \Lambda_{1, \; \mu -2}, 
\;\; e_{2} =
\begin{pmatrix}
0  &   0  \\
0  &  1
\end{pmatrix}
\in  \Lambda_{2, \; \mu - 2}.
\end{gather}    
The invertible matrix $L$ is used to produce the idempotent
\begin{gather}    
P  :=  L e_{1} L^{-1} =
\begin{pmatrix}
S_{0}^{2}  &   S_{0} (1 + S_{0}) B  \\
S_{1}A  &  1 - S_{1}^{2}
\end{pmatrix}
\in \mathit{Idemp}_{2n}(\Lambda_{1, \;\mu - 2}).
\end{gather} 
with
\begin{equation*}
j_{1} (P) = 
\begin{pmatrix}
0 &0  \\
0 & 1
\end{pmatrix}
.
\end{equation*}
\par
The idempotent $P$ is used to construct the double matrix idempotent  $\mathbf{P}_{U}$
\begin{gather}  
\mathbf{P}_{U}  :=
\begin{pmatrix}
(S_{0}^{2}, \; 0)  &   (S_{0} (1 + S_{0}) B, \; 0)  \\
(S_{1}A, \; 0)   &  (1 - S_{1}^{2}, \; 1 )
\end{pmatrix}
\in \mathit{Idemp}_{2n}(\Lambda_{1, \mu - 2}) \oplus  
\mathbb{I}demp_{2n}(\Lambda_{2, \mu - 2}).
\end{gather}   
with
\begin{equation*}
\pi \;(\mathbf{P}_{U}) =
\begin{pmatrix}
0 & 0 \\
0 & 1
\end{pmatrix}
\in \mathbb{I}demp (\Lambda_{2, \mu}).
\end{equation*} 
The idempotent $\mathbf{P}_{U}$  satisfies    
$(j_{1, \ast} - j_{2, \ast})( \mathbf{P}_{U}) = 0$. Therefore, 
$\mathbf{P}_{U} \in \mathbf{M}_{2n}(\Lambda_{\mu - 2})$.
\end{definition}   
\begin{definition} \ref{connectinghomomorphismsecondform}
\emph{Connecting homomorphism - second form.}   
 (for  $K$-theory see \cite{Connes_Moscovici}).   
\par
For any $[U] \in T_{1}^{loc} (\Lambda^{'})$ we define the \emph{connecting homomorphism} 
$\partial_{II} : T_{1}^{loc}(\Lambda') \longrightarrow T_{0}^{loc}(\Lambda)$ by
\begin{equation}   
\partial_{II} [U] := [ \mathbf{P}_{U}]  - [(e_{2}, e_{2})] \in  T_{0}^{loc} (\Lambda).
\end{equation}   
\end{definition}   
In this case it is easy to check that $\partial$ is compatible with the equivalence relation $\sim_{\mathbb{O}}$.
Indeed, any element $\xi = u \oplus u^{-1} \in \Lambda^{'}_{\mu}$ has an invertible lifting in 
$\Lambda_{1, \; \mu}$, see \S \ref{liftingO}. For this reason, both $A$  and  $B$ may be chosen to be inverse one to each other. Threfore, $S_{0} = 0$ and $S_{1} = 0$ and hence
\begin{gather*}
P = 
\begin{pmatrix}
0 & 0 \\
0 & 1
\end{pmatrix},
\;\;\; 
\partial [U] \; = 0.
\end{gather*}
\subsection{Connecting homomorphism and stabilisation}    

In this sub-section we show how $\partial$ depends on stabilisations.
In the proof of the exactness,  Theorem \ref{exactsequence} .{\bf -iii)},     
we will need to know how the connecting homomorphism, Definition  \ref{connectinghomomorphismsecondform}, behaves with respect to stabilisations. 
For this purpose we consider a more general situation than that considered in the previous section \S \ref{connectinghomomorphismsecondform}. 
\par
Let  $U \in \mathbb{GL}_{m+n} (\Lambda')$ and 
\begin{equation}   
e_{(0,n)} =
\begin{pmatrix}
0 & 0 \\
0 & 1_{n}
\end{pmatrix} 
\in \mathbf{M}_{m+n} (\Lambda')
\end{equation}    
be such that the diagram
\begin{equation}  \label{commutativityhypothesis}
\begin{CD}  
\Lambda^{'m+n}  @  >  {  U }   >>   \Lambda^{'m+n}  \\
@VV  { e_{(0,n)} }  V   @    VV { e_{(0,n)} }  V \\
\Lambda^{'m+n}  @  > { U  } >>   \Lambda^{'m+n} 
\end{CD}
\end{equation}   
is commutative. This condition will be needed to show that 
$\mathbf{P}_{U} \in \mathbb{I}demp_{m+n}(\Lambda)$, formula 
(\label{21.118}).   
\par
The case discussed in 
\S \ref{connectinghomomorphismsecondform}   
corresponds  in this subsection to $m = 0$ .  
\par
We proceed as in 
\S \ref{connectinghomomorphismsecondform}. 
Let $A, \; B \in  \mathbb{GL}_{m+n} (\Lambda_{1, \mu})$  be liftings of $U$, resp. $U^{-1}$, in $\mathbf{M}_{m+n, \mu} (\Lambda_{1})$. Such liftings exist because we assume $j_{1}$ is surjective.
\par
With $A$ and $B$ one associates $S_{0}= 1 -  BA \in \mathbf{GL}_{m+n} (\Lambda_{1, \mu})$ and 
$S_{1}= 1 -  AB \in \mathbb{GL}_{m+n} (\Lambda_{1, \mu})$.
The matrices $S_{0}$, $S_{1}$ satisfy 
\begin{equation}   
j_{1, \ast } (S_{0}) = j_{1, \ast } (S_{1}) = 0.
\end{equation}  
With these matrices one associates the invertible matrix    
\begin{gather}  
L =
\begin{pmatrix}
S_{0}  &  - (1 + S_{0}) B  \\
A  &  S_{1}
\end{pmatrix}
\in \mathbf{GL}_{2(m+n)}(\Lambda_{1, \mu - 2});
\end{gather}   
the inverse of the matrix $L$ is 
\begin{gather}  
L^{-1} =
\begin{pmatrix}
S_{0}  &   (1 + S_{0}) B  \\
- A  &  S_{1}
\end{pmatrix}
\in \mathbf{GL}_{2(m+n)}(\Lambda_{1, \mu - 2}).
\end{gather}  
Let $e_{1} $ be the idempotent
\begin{gather}  
e_{1} = 
\begin{pmatrix}
e_{(0,n)} & 0\\
0 & 0
\end{pmatrix} \;\; \in \mathbb{GL}_{2(m+n)}(\Lambda_{1, \mu }).
\end{gather}   
and

\begin{gather}  
e_{2} = 
\begin{pmatrix}
0 & 0\\
0 & e_{(0,n)}
\end{pmatrix} \;\; \in \mathbb{GL}_{2(m+n)}(\Lambda_{1, \mu }).
\end{gather}   
The invertible matrix $L$ is used to produce the idempotent
\begin{gather}  
P_{U}  :=  L e_{1} L^{-1} =
\begin{pmatrix}
S_{0} \;e_{(0,n)}  \; S_{0} &   S_{0} \; e_{(0,n)}  \; (1 + S_{0}) B  \\
A \;e_{(0,n)}  \;S_{0} &  A \;e_{(0,n)}  \; (1 + S_{0}) B
\end{pmatrix}
\in \mathit{Idemp}_{2(m+n)}(\Lambda_{1, \mu - 2}).
\end{gather}   
\begin{definition}  
\begin{equation}   
 R(U) := P_{U} - e_{2}.   
 \end{equation}   
 \end{definition}  
The idempotent $P_{U}$ is used to construct the double-matrix idempotent
\begin{gather}   
\mathbf{P}_{U}  :=
\begin{pmatrix}
(S_{0}  \;e_{(0,n)}  \; S_{0}, \; 0)  &   (S_{0} \; e_{(0,n)}  \;(1 + S_{0}) B, \; 0)  \\
(A \; e_{(0,n)}  \;S_{0}, \; 0)   & (A \; e_{(0,n)}  \;  (1 - S_{0}) B, \; e_{(0,n)} )
\end{pmatrix}
\\
\in \mathbf{I}demp_{2(m+n)}(\Lambda_{1, \mu - 2} \oplus    \Lambda_{2, \mu - 2}).
\end{gather}    
\par
The matrix   $\mathbf{P}_{U}$  is an idempotent in $\mathbf{M}_{2(m+n)}(\Lambda_{\mu - 2})$.
We may verify directly that  $(j_{1, \ast} - j_{2, \ast}) \mathbf{P}_{U} = 0$. Indeed, $j_{1 \ast}$ is a ring homomorphism
and  $j_{1 \ast} (S_{0}) = j_{1 \ast} (S_{1}) = 0$; finally,  the hypothesis 
(\ref{commutativityhypothesis}) 
gives
\begin{equation}  \ref{21.118}  
j_{1}(A \; e_{(0,n)}  \;  (1 - S_{0}) B = U e_{(0,n)} U^{-1} = e_{(0,n)} = j_{2} (e_{(0,n)} ).
\end{equation}   
 Therefore
$\mathbf{P}_{U} \in \mathbf{M}_{2n}(\Lambda_{\mu - 2})$.
The fact that the matrix $\mathbf{P}_{U}$ is an idempotent in $\mathbf{M}_{2(m+n)}(\Lambda_{\mu -2})$ follows from the fact that $\mathit{P} $ and $e_{2}$  are idempotents along with the discussion above. 
\begin{definition}   \label{connectinghomomorphismthirdform}  
\emph{Connecting homomorphism - third form.} 
\par
We suppose the assumptions and constructions of 
\S \ref{connectinghomomorphismsecondform} 
above are in place. 
For any $[U] \in T_{1}^{loc} (\Lambda^{'})$ one defines the \emph{connecting homomorphism} 
$\partial : T_{1}^{loc}(\Lambda') \longrightarrow T_{0}^{loc}(\Lambda)$ by
\begin{gather}   
\partial [U] := [ \mathbf{P}_{U}]  - [(e_{2}, e_{2})] =  \\
\begin{pmatrix}
(0 , \; 0)  &   (0, \; 0)  \\
(0, \; 0)   & (A e_{(0,n)}  B, \; e_{(0,n)}  )
\end{pmatrix} 
 - [(e_{2}, e_{2})] 
\in  T_{0}^{loc} (\Lambda).  
\end{gather}  
\end{definition}   
\par
It remains to follow up how $\mathbf{P}_{U}$ depends of the choice of the lifts $A$ and $B$ of $U$ and $U^{-1}$. 
A different choice of $A$ and $B$ has the effect of modifying the matrix 
$L$. We will show that if
$A'$ and $B'$ are two such different lifts and $L'$ is the corresponding matrix, then
\begin{equation}    
L' = \tilde{L} \;  L,     \hspace{0.5cm}  with   \hspace{0.2cm}  \tilde{L} \in \mathbf{GL}_{2(m+n)} (\Lambda)
\end{equation}   
and hence the corresponding idempotents $\mathbf{P}_{U} := L e_{1} L^{-1}$,
$\mathbf{P}'_{U} := L' e_{1} L'^{-1}$ are conjugate.
 \par
 To better organise the computation, we change the liftings one at the time. 
 \par
 We begin with $A$. Let $\tilde{A} = A + T$ with $j_{1}(T) = 0$. Let $\tilde{S}_{0} = 1 - B \tilde{A}$, 
 $\tilde{S}_{1} = 1 -  \tilde{A} B$,  $\tilde{L}$,  $\tilde{P}$ and $\tilde{\mathbf{P}}_{U}$ be the corresponding elements.
 A direct computation gives
  
 \begin{gather}    
 \tilde{L} L^{-1} = 
 \begin{pmatrix}
1 - B T &  - B T B  \\
     T   &  1 +  T B
\end{pmatrix}   
\end{gather}     
or
\begin{gather}  
 \tilde{L}  = 
 \begin{pmatrix}
1 - B T &  - B T B  \\
     T   &  1 +  T B
\end{pmatrix}  
 L.
\end{gather}  
We know that $\tilde{L}$ and $L$ are invertible matrices; therefore the RHS of (\ref{21.123})   
is an invertible matrix.
\par
The corresponding idempotent $\tilde{P}$ is
\begin{gather}   
\tilde{P} = \tilde{L}  . e_{1} . \tilde{L}^{-1} =    
\begin{pmatrix}
1 - B T &  - B T B  \\
     T   &  1 +  T B
\end{pmatrix}
 L. e_{1} . L^{-1}
\begin{pmatrix}
1 - B T &  - B T B  \\
     T   &  1 +  T B
\end{pmatrix} 
^{-1} = \\        
= \begin{pmatrix}   
1 - B T &  - B T B  \\
     T   &  1 +  T B
\end{pmatrix}
 P
\begin{pmatrix}
1 - B T &  - B T B  \\
     T   &  1 +  T B
\end{pmatrix} 
^{-1}
\end{gather}   
and furthermore
\begin{gather}   
\tilde{\mathbf{P}}_{U} =
(
\begin{pmatrix}
1 - B T &  - B T B  \\
     T   &  1 +  T B
\end{pmatrix}, \; 1) \;
 \mathbf{P}_{U} \;
( \begin{pmatrix}
1 - B T &  - B T B  \\
     T   &  1 +  T B
\end{pmatrix}, \; 1)
^{-1}.
\end{gather}   
Therefore, $[\tilde{\mathbf{P}}_{U}] = [\mathbf{P}_{U}] \in T_{0}^{loc}(\Lambda)$.
\par
It remains to see what happens if $A$ remains unchanged and the lifting B is changed. Let $\tilde{B} = B + H$, with
$j_{1} (H) = 0$.  Let $\tilde{\tilde{L}}$,  $\tilde{\tilde{P}}$ and  $\tilde{\tilde{\mathbf{P}}}_{U}$ be the corresponding
matrices.  A direct computation gives 
\begin{gather}   \label{21.127}   
\tilde{\tilde{L}} . L^{-1} =    
\begin{pmatrix}  
1 + \Delta_{11}  &  \Delta_{12}  \\
\Delta_{21}  &  1 + \Delta_{22}
\end{pmatrix}
  \in \mathbf{M}_{2n} ( \Lambda_{\mu -4}) 
\end{gather}    
where 
\begin{align*}
\Delta_{11}  =  &  HA - HAHA - BAHA     \\
\Delta_{12}  =  & -2H + HAB + HAH + BAH - BAHAB - HAHAB  \\
\Delta_{21}   =  & AHA    \\
\Delta_{22}  =   & - AH + AHAB.
\end{align*} 

The RHS of (\ref{21.127})   
is a product of invertible matrices; therefore, it is an invertible matrix. Proceeding as above we get
\begin{gather}   
\tilde{\tilde{\mathbf{P}}}_{U} =  
( \begin{pmatrix}  
1 + \Delta_{11}  &  \Delta_{12}  \\
\Delta_{21}  &  1 + \Delta_{22}
\end{pmatrix}
, \; 1)
\; \mathbf{P}_{U} \; 
( \begin{pmatrix}  
1 + \Delta_{11}  &  \Delta_{12}  \\
\Delta_{21}  &  1 + \Delta_{22}
\end{pmatrix},
\; 1)
^{-1}.
\end{gather}  
This completes the discussion about the choice of the liftings $A$ and $B$.
 

\part{Topological index and analytical index. Reformulation of index theory.}
\section{Level I: Index theory at the $T^{loc}_{\ast}$-theory level.}
\begin{definition}
Let  $\mathcal{J}  \subset \mathcal{A}$  be a localised bi-lateral ideal of the unital ring $\mathcal{A}$.
There correponds the short ring exact sequence     
\begin{equation}  \label{exactsequence20}
0 \longrightarrow  \mathcal{J}_{\mu} \overset{\iota}{\longrightarrow}  \mathcal{A}_{\mu} 
  \overset{\pi}{\longrightarrow}   \mathcal{A}_{\mu}/\mathcal{J}_{\mu}  \longrightarrow  0.
\end{equation}
\begin{definition}
The ring $\mathcal{A}_{\mu} / \mathcal{J}_{\mu}$  is the analogue of the Calkin ring. 
If $U \in \mathcal{A}_{\mu}$ then $\pi (U) \; \in  \mathcal{A}/\mathcal{J}$  is called
the \emph{symbol} of $U$.
\end{definition}
\par
Consider the 6-term exact sequence in $T^{loc}$-theory (\ref{secondexactsequence})  
\begin{gather}   \ref{secondexactsequence}  
T_{1}^{loc}(\mathcal{A}, \mathcal{J}) \overset{i_{\ast}}{\longrightarrow}
T_{1}^{loc}(\mathcal{A}) \overset{\pi_{\ast}}{\longrightarrow}
T_{1}^{loc}(\mathcal{A}/\mathcal{J}) \overset{\partial}{\longrightarrow} \\
T_{0}^{loc}(\mathcal{A}, \mathcal{J}) 
\otimes \mathbb{Z}[\frac{1}{2}]
\overset{i_{\ast}}{\longrightarrow} 
T_{0}^{loc}(\mathcal{A}) 
\otimes \mathbb{Z}[\frac{1}{2}]
\overset{\pi_{\ast}}{\longrightarrow} 
T_{0}^{loc}(\mathcal{A}/\mathcal{J})
\otimes \mathbb{Z}[\frac{1}{2}].
\end{gather}
\par
\begin{enumerate}
\vspace{0.3cm}
\item
\emph{Topological index}  of $u$ is
\begin{equation}
Top^{T}Index \; (u) := \; \delta [u] \in T_{0}^{loc}(\mathcal{A}) \otimes \mathbb{Z}[\frac{1}{2}],
\end{equation}
where  $[u] \in  T_{1}^{loc} (\mathcal{A}/ \mathcal{J})$.
\vspace{0.5cm}
\item
\emph{Analytical Index} of $u$ is
\end{enumerate}
\begin{equation}
An^{T}Index \; (u) := \; \mathbf{R} \;( \delta  [u] ) \; =  \;  \delta_{II}  [u] 
\in T^{loc}_{0} (\mathcal{A}) \otimes \mathbb{Z}[\frac{1}{2}]
\end{equation}
where $\delta_{II}$ is the second definition of boundary map,   see  Definition
\ref{connectinghomomorphismsecondform}.
\end{definition}
\vspace{0.3cm}
{\bf Case 1.}
\begin{problem}
Define significant classes of extensions (\ref{exactsequence20})  
\begin{equation}
\Delta (\mathcal{A}, \mathcal{J}) \;:=\; Top^{T}Index \;(u) \; - An^{T}Index (u)  
\end{equation}
can be computed.
\end{problem}
\vspace{0.5cm}
{\bf Case 2.}
\begin{conjecture}  
\vspace{0.1cm}
-1) Let $M^{2l}$  be a closed compact quasi-conformal manifold. Let $\Omega^{\ast} (M)$ be the 
algebra of differential forms on $M^{2l}$. Let $H$ be the Hilbert space of $L_{2}$-forms of degree  $l$. 
Let 
 be the exact sequence associated to the short exact sequence 
$$
  0  \longrightarrow  \mathcal{L}^{(1, \infty)}  \overset{\iota}{\longrightarrow}    \Psi Diff  \overset{\pi}{\longrightarrow}
    \Psi Diff  /\mathcal{L}^{1, \infty}  \longrightarrow  0.   
$$
For any  $ u \in  \mathbb{GL}_{N}(\Psi Diff  / \mathcal{L}^{1, \infty})$ one has
\begin{equation}
 \;(Top^{T}Index) \; (u) \otimes {\mathbb{Q}}= \; (An^{T}Index) \; (u)  \otimes {\mathbb{Q}}.
\end{equation}
\vspace{0.2cm} 
\par
-2)  Let $A$  be an elliptic pseudo-differential operator on $M$. Let
$u = \pi  (A)  \in  \mathbb{GL}_{N} (\mathcal{A} / \mathcal{J})$ be the image of $A$ in the quotient space. 
Then
\begin{gather}  \label{20.8}   
An^{T}Index \; (u) \otimes \mathbb{Q}  \;=\; Top^{T}Index \; (u) \;  \otimes \mathbb{Q} = \\  
 \;=\; \frac{1}{(2\pi i)^{q}} \frac{q!}{(2q)!}(-1)^{dim M} \; Ch  (u)  \; \cap \; [T^{\ast}M].
\end{gather}
$Ch  (u)$ in the formula  (\ref{20.8}) 
 is the periodic cyclic homology of $u$.  
\vspace{0.2cm}
\par
-3) Let $A$  be a pseudo-differential elliptic operator on the quasi-conformal manifold $M$. 
Let $u$  be its \emph{classical} symbol.  Produce the residue
operator $\mathbf{R}(u)$  
($K_{1}$ is replaced by  $T^{loc}_{1})$.
Let $f \in  C_{AS}^{q} (M)$  be an Alexander - Spanier co-cycle; let $[f]$  be its co-homology class.
Let $\tau (T^{\ast} (M))$  be the \emph{Todd class} of $M$.
\par
Then   
\begin{gather}
An^{T}Index \; (u)  \otimes_{\mathbb{Q}} [f] \;=\; Top^{T}Index \; (u) \;  \otimes_{\mathbb{Q}} [f] \; = \\  
 \;=\; \frac{1}{(2\pi i)^{q}} \frac{q!}{(2q)!}(-1)^{dim M} \;Ch  (u)  \; \cup \; \tau  (T^{\ast} (M)) \; \cap \; [T^{\ast}M].
\end{gather}
\par
\end{conjecture}   

\section{Level II:  Index Theory in \emph{local} periodic cyclic homology.}

\begin{definition}
Let $\mathcal{A}_{\mu}$  be a localised ring.  Consider the exact sequence   \ref{exactsequence20}.
Let $u \in T^{loc}_{1}(\mathcal{A}_{\mu}/ \mathcal{J}_{\mu})$.
\begin{enumerate}
\vspace{0.3cm}
\item
\emph{Topological index}  of $u$ is
\begin{equation}
Top^{Ch}Index \; (u) := \; Ch \delta [u] \in C^{loc, per, \lambda}(\mathcal{A}),
\end{equation}
\vspace{0.5cm}
\item
\emph{Analytical Index} of $u$ is
\end{enumerate}
\begin{equation}
An^{Ch}Index \; (u) := \; Ch_{ev} \;( \delta  \;[u] ) \; =  \;  Ch_{ev}( \delta_{II}  \;[u])
\in H^{loc, per, \lambda}_{0} (\mathcal{A}),
\end{equation}  
see  Definition  \ref{connectinghomomorphismsecondform}. 
\item
The topological index could be defined in a different way. One could consider the connecting
homomorphism $\delta_{\lambda}$  in the local cyclic periodic homology instead of the connecting
homomorphism in the $T^{loc}_{\ast}$ exact sequence. This leads to the 
\emph{topological index}   
\begin{equation}
Top^{Ch_{\lambda}}Index \; (u) \;:= \; Ch \delta_{\lambda} [u] \in H^{loc, per, \lambda}_{ev}(\mathcal{A}),
\end{equation}
\end{definition}
\vspace{0.3cm}
{\bf Case 1}
\begin{problem}  
\par
Define significant classes of extensions (\ref{exactsequence20})  
for which the difference
\begin{equation}
\Delta^{Ch} (\mathcal{A}, \mathcal{J}) \;:= \; Top^{Ch}Index \;(u) \; - \; An^{Ch}Index \;(u)  
\end{equation}
can be computed.
\end{problem}
{\bf  Case II.}
\begin{conjecture} 
Let $M$  be a quasi-conformal closed manifold. Let $u \in  \mathbb{GL}_{N} (\mathcal{A}_{\mu}/ \mathcal{J}_{\mu})$, where $(\mathcal{A}_{\mu})$  is the algebra of  pseudo-differential  operators on $M$  and  $\mathcal{J}) = \mathcal{L}^{(\frac{1}{n}, \infty)}$,    localised by the support of operators about the diagonal.  
Then
\begin{equation}
Top^{Ch}Index (u) \;=\; An^{Ch}Index (u).
\end{equation}
\end{conjecture}

\section{Level III:  Index theory restricted at the diagonal.}
This situation applies only when the local periodic cyclic homology has a limit to the diagonal.
It depends on the regularity of the structure.  The classical index theorems belong to this class.
\par
\vspace{1cm}


\part{Noncommutative Topology}   
\noindent 
\abstract{
We intend to produce a theory which generalises topological spaces;  we call it \emph{non-commutative topology}. In non-commutative \emph{differential} geometry the \emph{basic homology theory} is the \emph{periodic cyclic homology}, based on the bi-complex $(b, B)$. In non-commutative topology this structure will be replaced by the bi-complex $(\tilde{b}, d)$; the boundary $\tilde{b}$ is called \emph{modified Hochschild boundary}.  These ideas combine A. Connes' work \cite{Connes_Book}  with ideas of the articles by Teleman N. and Teleman K.
\cite{Teleman_1966},  \cite{Teleman_K}, \cite{Teleman_arXiv_I},  
\cite{Teleman_arXiv_II}, \cite{Teleman_arXiv_III}, \cite{Teleman_arXiv_IV},   \cite{Teleman_arXiv_V}.
 }  

\section{Modified Hochschild homology}   

The Alexander - Spanier co-homology uses solely the topology of the space; it does not require any kind of analytical regularity.  We use the Alexander - Spanier construction and the definition of local periodic cyclic co/homology as the departure point for non-commutative topology.  
Recall we extended the  Alexander - Spanier co-homology to arbitrary \emph{localised rings},  
see \S 21. 
\par
As an application we compute the local modified periodic cyclic homology of the topological algebra of \emph{smooth} functions, Theorem \label{theorem1.24}, 
of the Banach agebra of \emph{continuous} functions, Theorem \label{theorem2.24},  
and of the algebra of arbitrary functions, Theorem \ref{theorem3.24}, on a smooth manifold.
These results are significant because it is known that the Hochschild and (periodic) cyclic homology of \emph{Banach algebras}  are either trivial or not interesting, see Connes \cite{Connes_IHES}, \cite{Connes_Book}, \cite{Connes_Moscovici}. 
\emph{Entire cyclic cohomology}, due to  Connes \cite{Connes_Book}
 (\cite{Connes_Moscovici}, \cite{Connes_Gromov_Moscovici})
gives a different solution to the problem of defining the cyclic homology of Banach algebras.
  The chains of the entire cyclic cohomology are elements of the infinite product $(b, B)$ which satisfy a certain bi-degree asymptotic growth condition.  
Connes constructed a Chern character  of $\theta$-summable Fredholm modules with values in the entire cyclic cohomology, 
see Connes \cite{Connes1a}.

\section{The idempotent $\Pi$.}   
Recall the operator $\sigma$ was introduced in \S  2.4.10  It is well defined on the whole complex  $C_{\ast}(\mathcal{A})$ 
\begin{equation}  
 d b + bd = 1 - \sigma,
\end{equation}    
where $d$ is the non-localised Alexander - Spanier co-boundary and $b$ is the Hochschild boundary.
\par
The main properties of  $\sigma$ are
\begin{enumerate}
\item 
$\sigma$ commutes both with $d$ and $b$.  
\item
$\sigma$ is a chain homomorphism both in the Alexander-Spanier and in the Hochschild complex. Hence, the range of the operator $\sigma$, and its powers, are sub-complexes both in the Alexander - Spanier and Hochschild complexes. 
\item
The range of the homomorphism $\Pi$  consists of non-degenerate chains.
\item 
$\sigma$ is homotopic to the identity. Therefore, the inclusions of these sub-complexes into the Alexander - Spanier, resp. Hochschild, complexes induce isomorphisms between  their homologies.
\end{enumerate} 
\par
The operator $\Pi_{(k)}$ is defined by the formula  
\begin{equation}  
\Pi_{(k)} := (1 - bd) \sigma_{(k)}^{k}.
\end{equation}   
The opertor $\Pi_{(k)}$ has the properties
\begin{enumerate}
\item
$\Pi_{(k)}$ is an idempotent. Let $\{ \tilde{C}_{\ast}({\mathcal{A}}) \}$ be its range.
Hence 
\begin{equation}   
(1 - bd) \sigma_{(k)}^{k} = 1 \;\; on \; \tilde{C}_{k}(\mathcal{A})
\end{equation}
\item
$\Pi_{(k)}$ commutes with $d$, $b$, and hence with $\sigma$.  
Therefore
\item
$\{ \tilde{C}_{\ast}({\mathcal{A}})$    is a $b$ and $d$ sub-complex
\item
The mapping  $\Pi_{(k)}$  induces isomorphisms  between the $b$, resp. $d$, homologies.
\par
The sub-complex $\{   (1 - \Pi_{k})C_{k}(\mathcal{A}) \}_{k \in \mathbb{Z}} $ is acyclic.
\end{enumerate}
\section{Modified Hochschild homology.}     
On the range of the idempotent $\Pi$, i.e.   $\{\tilde{C}_{\ast}({\mathcal{A}}) \}$,   one has the identity
\begin{equation}   
\sigma^{k} = 1 + bd\sigma^{k}.
\end{equation}    
From this formula we get
\begin{proposition}   
On the complex $\{\tilde{C}_{\ast}({\mathcal{A}}) \}$  one has the identity
\begin{equation}  \label{identity.24}   
 0 = \sum_{k=1}^{n+1}  (-1)^{k-1} C_{n+1}^{k} (bd)^{k}
+ \sum_{k=1}^{n} (-1)^{k-1} C_{n}^{k}  (db)^{k}.
\end{equation}
\end{proposition}  
\begin{proof}  
The relation  (\ref{identity.24}) 
is obtained from the formula 
\begin{equation}
 1 = ( 1 - bd ) \sigma^{n}.
\end{equation}
by making the substitution  $\sigma = 1 - (db + bd)$.
For more details see  \cite{Teleman5}.
\end{proof}   

Formula (\ref{identity.24}) 
suggests to introduce on the spaces $C_{\ast} (\mathcal{A})$
the operators $\tilde{b}$,  $\tilde{d}$.
\begin{definition}   
The operators $\tilde{b}_{n}$ and $\tilde{d}_{n}$ acting on $C_{n}(\mathcal{A})$ are defined by the formulae
\begin{gather}  \label{tildeb}  
\tilde{b}_{n} :=  b \sum_{k=1}^{n} (-1)^{k-1} C_{n}^{k}  (db)^{k-1} \\
\tilde{d}^{n} :=  d \sum_{k=1}^{n} (-1)^{k-1} C_{n+1}^{k}  (bd)^{k-1}.
\end{gather}
\end{definition}   
\begin{proposition}  
The operators $\tilde{b}$ and $d$ anti-commute in  $\tilde{C}(\mathcal{A})$;
the same relation holds for the operators $b$  and $\tilde{d}^{n}$
\begin{gather}
0 = \tilde{b}_{n+1} d_{n} + d_{n-1} \tilde{b}_{n} \\
0 = \tilde{d}_{n-1} b_{n} + b_{n+1} \tilde{d}_{n}.
\end{gather}
\end{proposition}  
\begin{definition}  
  The homology of the complex  
$\{   C_{\ast} (\mathcal{A}, \; \tilde{b})  \}_{\ast}$ is called \emph{modified Hochschild homology}. 
\end{definition}
The expression of the operator $\tilde{b}_{n}$  contains the factor $b$ both to the left and to the right. 
\begin{definition}  
-1) $ \tilde{b} \tilde{b} = 0$ and hence $\tilde{b}$ is a boundary operator. 
\par
-2) The complex  $\{ C_{\ast}(\mathcal{A}), \; \tilde{b} \}_{\ast}  \}  $ is called \emph{modified Hochschild complex}. It makes sense even for non-localised algebras.
\par
The homology of the modified Hochschild complex  is called \emph{modified Hochschild homology}.
\par
 -3)  We assume $\mathcal{A}$ to be a localised ring;  the
 Alexander - Spanier and the Hochschild complex are localised.  On the space of normalised chains $(\tilde{b}, d)$  introduces a bi-complex structure. The boundary $\tilde{b}$  defined on the space 
 of normalised chains is called  \emph{local modified Hochschild homology}. 
 \par
 The homology of the bi-complex $(\tilde{b}, d)$ is called \emph{local modified periodic cyclic homology}; it  is $\mathbb{Z}_{2}$-graded.
 \end{definition}
 The definition of the local modified periodic cyclic homology is analogues to the definition of periodic cyclic homology,  comp. A. Connes \cite{Connes_Book}, see also J. - L. Loday \cite{Loday} Sect. 5.1.7., pag.159.
\par
More specifically, given the localised ring $\mathcal{A}$, one has
\begin{equation}
(\tilde{b}, d)^{\lambda, per}( {\mathcal{A}}) = \{ C_{p,q}(\mathcal{A}) \}_{ (p,q) \in \mathbb{Z} \times \mathbb{Z} },
\end{equation}
where  $C_{p,q}(\mathcal{A}) := \otimes^{p-q} \mathcal{A}$, for  $q \leq p$.
The boundary maps are 
$$
\tilde{b}:  C_{p,q} \longrightarrow C_{p,q-1}
$$
and
$$
d:  C_{p,q} \longrightarrow C_{p+1,q};
$$

\begin{proposition}      \ref{tildeb}
For any ring $\mathcal{A}$, the homology of the modified Hochschild complex contains the Hochschild homology.
\end{proposition}  
\begin{proof}  
As said before, the operator $\tilde{b}$  contains the factor $b$ both to the left as well as to the right.
For this reason any Hochschild cycle is a $\tilde{b}_{n}$ cycle; for the same reason any $\tilde{b}$ boundary is a $b$ boundary.
\end{proof}  
\subsection{\emph{Local} modified periodic cyclic homology of the algebra of smooth functions.}  
\par
In this subsection we consider the Fr$\acute{e}$ch$\acute{e}$t topological algebra of \emph{smooth} functions over a \emph{compact} smooth manifold $M$.
We know  that the Hochschild boundary is well defined on germs at the diagonals and that the Hochschild homology depends only on the quotient complex, see   Teleman \cite{Teleman_CR}.
\par
The first two terms of the spectral sequence associated to the first filtration (with respect to
$d$ and then with respect to $\tilde{b}$) of the $(\tilde{b}, d)$ bi-complex are
$$
E^{1}_{p,q} = H_{p} (C_{\ast,q}, d) \cong H_{dR}^{q-p} (M)
$$
and
$$
E^{2}_{p,q}  = H_{q}(E^{1}_{p,\ast}, \tilde{b}) \cong H_{dR}^{q-p}(M).
$$
\par
In fact, the  $E^{1}$-term is the Alexander - Spanier co-homology of the complex of smooth chains.
The Alexander - Spanier chain complex contains the sub-complex 
$\sigma_{\ast} C_{\ast}  (\mathcal{A})$.  We know that this complex is co-homologous with the original Alexander - Spanier complex.  For this reason, 
for the computation of  $E^{2}$,
we may represent any element of $E^{1}$ by an element of the this sub-complex.   The term $E^{2} = \tilde{b}$-homology is a quotient group of a sub-group of $E^{1}$;  from this we get 
\begin{equation}  \label{AB1} 
dim \{\tilde{b}-homology \;of \;a \; quotient \;group \; of \; a \; sub-group \; of \; E^{1}  \} =   dim \; E^{2} \leq \; dim \; E^{1}.
\end{equation}
The term $E^{1}$, being the Alexander - Spanier co-homology of $M$, is isomorphic to the
periodic $b$-homology of the algebra $\mathcal{A}$. 
\begin{equation}   \label{AB2}  
dim \; E^{1} \; = \; dim \{ \; b-homology \;of \; C^{\infty}(M) \;\}
\end{equation}
On the other side,  the $\tilde{b}$-homology contains the $b$-homology,  
Proposition \ref{tildeb}   
\begin{equation}  \label{AB3}  
dim \;(b-homology) \;    \leq \; dim( \tilde{b}-homology).
\end{equation}
From the equations (\ref{AB1}), (\ref{AB2}) and (\ref{AB3})  
we get
\begin{equation}
E^{2} \; \cong \; H_{dR}^{\ast}(M).
\end{equation}
We have proved the following result.
\begin{theorem}   \ref{theorem1.24}   
The local modified periodic cyclic homology of the algebra of smooth functions on a \emph{compact} smooth manifold is isomorphic to the $\mathbb{Z}_{2}$-graded de Rham co-homology of the manifold.
\end{theorem}   
\section{Characteristic classes of idempotents.}
\begin{definition}  
Let $\mathcal{A}$  be a localised ring which contains the rational numbers. Let $ e \in \mathbb{M}_{N}(\mathcal{A})$ be an idempotent.
\par
The \emph{Chern character} of $e$ is the even local periodic cyclic homology class defined by 
\begin{equation}   \label{ChernCharacterTopologyCycle}  
Ch\;(\;e\;)  \;:=\; e  +  \frac{(-1)^{1}}{1 \;!} \;e \;(de)^{2}   +  \frac{(-1)^{2}}{2 \;!}   e \;(de)^{4}   +  \dots +   \frac{(-1)^{n}}{n \;!}     e \;(de)^{2n} +  \dots  
\end{equation}
\end{definition}
\begin{remark}  
It is important to notice that the Chern character defined above does not use the \emph{trace}.
\end{remark}  
\begin{theorem}   \label{ChernCharacterTopologyClass}
Let $\mathcal{A}$  be a localised ring.  Suppose the ring $\mathbb{K}$ contains 
$\mathbb{Q}$. Let $ e \in \mathbb{M}_{N}(\mathcal{A})$ be an idempotent. 
\par
Then the  RHS  of the equation (\ref{ChernCharacterTopologyCycle}) is a cycle in the bi-complex $(\tilde{b}, \; d)$. 
\end{theorem}  
\begin{proof}   
We compute $\tilde{b}\;e (de)^{2n}$. We have
\begin{gather}
\tilde{b}_{n}  (e \;(de)^{2n}) =
  b \sum_{k=1}^{2n}   (-1)^{k-1} C_{2n}^{k}  (db)^{k-1}  =  \\
                     b    C_{2n}^{1} + b \sum_{k=2}^{n}  (-1)^{k-1} C_{2n}^{k}  (db)^{k-1}  = \\
\tilde{b}_{2n} ( e  (de)^{2n}  ) =  2n \;e \;(de)^{2n-1}  +   
                       b      \sum_{k=2}^{2n}  \;(-1)^{k-1} C_{2n}^{k}  (db)^{k-2} (db) ( e (de)^{2n}  )    = \\ 
                                                         2n \;e \;(de)^{2n-1}  +   
                       b    \sum_{k=2}^{2n}  (-1)^{k-1} C_{2n}^{k}  (db)^{k-2}  d  (e \; de)^{2n-1}  )    =  \\
                                                        2n e (de)^{2n-1}  +   
                        b   \sum_{k=2}^{2n}  (-1)^{k-1} C_{2n}^{k}  (db)^{k-2}    (de)^{2n}  )    =  \\
                                                      2n e \;(de)^{2n-1}  +   
                            b   \sum_{k=2}^{2n}   (-1)^{k-1} C_{2n}^{k}  2^{k-2}  (de)^{2n}  )    =  \\
                                                      2n e \;(de)^{2n-1}  +   
                                 \sum_{k=2}^{2n}  (-1)^{k-1} C_{2n}^{k}  2^{k-2}  ( 2e - 1) (de)^{2n-1}  )    =   \\
                  (  2n  +    \sum_{k=2}^{2n}  (-1)^{k-1} C_{2n}^{k}  2^{k-1}   )   e \;(de)^{2n-1}  +   
                                \sum_{k=2}^{2n}  (-1)^{k} C_{2n}^{k}  2^{k-2}     (de)^{2n-1}.               
\end{gather}
\begin{lemma}  \label{lemma24}  
-1) 
\begin{equation}
 2n  +    \sum_{k=2}^{2n}  (-1)^{k-1} C_{2n}^{k}  2^{k-1}  \; = \; 0,
\end{equation}
\par
-2)
\begin{equation}
 \sum_{k=2}^{2n}  (-1)^{k} C_{2n}^{k}  2^{k-2}  \; = n.
\end{equation}
\end{lemma}   
\begin{proof}  
-1)  
\begin{gather}
 2n  +    \sum_{k=2}^{2n}  (-1)^{k-1} C_{2n}^{k}  2^{k-1}  \; = 
 2n - \frac{1}{2}  \sum_{k=2}^{2n}  (-1)^{k} C_{2n}^{k}  2^{k}  \; = \\
 2n -  \frac{1}{2} [  (1 - 2)^{2n} -  (1 - 2.2n) ]\; = \; 2n - \frac{1}{2} [  1 - 1 + 4n  ]  \;= \; 0.
\end{gather}
-2)
\begin{gather}
\sum_{k=2}^{2n}  (-1)^{k} C_{2n}^{k}  2^{k-2}  \; = 
\; \frac{1}{4} \sum_{k=2}^{2n}  (-1)^{k} C_{2n}^{k}  2^{k}  \; = \;
\; \frac{1}{4} [ \;(1 - 2)^{2n} -  (1 - 2n . 2) \; ] \;=\; n.
\end{gather}
\end{proof}   
From Lemma \ref{lemma24} we get
\begin{equation}
\tilde{b}_{2n} \; (\;e (de)^{2n}\;) \; = \; n \;(de)^{2n-1}.
\end{equation}
\end{proof}
On the other side
\begin{equation}
d \; (\;e \;(de)^{2n-2}\;) \; = \;(de)^{2n-1}.
\end{equation}
These prove the theorem.
\par

\begin{definition}  \label{characteristic subcomplex} 
Let $\Delta \subset \tilde{C}_{\ast}(\mathcal{A})$  be the $\mathbb{K}$ sub-module generated by all chains
\begin{equation}
e (de)^{n},     \;\;\;   (de)^{n},  \;\;
b\;(e (de)^{n}),      \;\;    b\;(de)^{n}, \;\;   
B(\; e (de)^{n}\;), \;\;  B\;(de)^{n},   \;\;               
d(\;e (de)^{n}),  \;\; d \;(de)^{n} 
\end{equation}
where $n \in \mathbb{N}$ and $e \in \mathbb{M}_{\ast}(\mathcal{A})$ is an arbitrary idempotent.
\par
$\Delta$ is called  \emph{characteristic sub-complex}. 
\end{definition}   
\begin{proposition}  
$\Delta$  is a sub-complex both in the $(b, B)$ and $(\tilde{b}, d)$ complexes.
\end{proposition}  
\par
Although not all relations below are used in this chapter, we provide them for the benefit of the reader.

\begin{equation}  
B  (a_{0} da_{1}  \dots da_{n}) \; =\; \sum_{j=0}^{j = n}   (-1)^{nj} 
da_{j} \dots da_{n}  da_{0} \dots da_{j-1}.
\end{equation}    
On the same space the formulas hold
\begin{equation}  
B^{2}\;=\;0, \;\;  bB + Bb \; = \; 0.
\end{equation}  
\par
The following relations hold on normalised local/non-local chains. Although not all of them are necessary in what folllows, we mention them for the benefit of the reader, see J. M. Garcia - Bondia,   H. Figueroa,  J. C. Varilly \cite{Bondia_Figueroa_Varilly},  Pg.  447
\begin{gather}
b(\;e (de)^{2n}\;)\; = \; \;e (de)^{2n-1}\;,      \;\;   b(\;e (de)^{2n-1}\;) \;=\;0  \\   
b\; (de)^{2n} \;=\; (2e-1)(de)^{2n-1}\;,  \;\;\   b\;   (de)^{2n-1} \;=\; 0 \\ 
B(\; e (de)^{2n}\;)\; = \; (2n+1) (de)^{2n+1}, \;\;  B(\;    (de)^{2n}\;)\; = \; 0  \\  
B\;    (e(de)^{2n-1}) \;= \; 0,                          \;\;  B\;    (de)^{2n-1} \;= \; 0,  \\ %
d(\;e (de)^{n}\;)\; = \; (de)^{n+1},               \;\;   d\;   (de)^{n}\; = \; 0.
\end{gather}
The following relation is a property of the idempotent $\Pi_{p}$: it is  needed next
\begin{equation}  \label{Bd} 
\Pi_{p+1} \;(B \;x) = d \; \Pi_{p} \; ( x ).       
\end{equation}
 \begin{theorem}   \ref{ChernCharacterTopologyCycle}   
Let $M$ be a compact smooth manifold and $e$  an idempotent of the matrix algebra over $\mathcal{A} = C^{\infty}(M)$.
\par 
Then the Chern character  $Ch(\;e\;)$  defined by the formula  (\ref{ChernCharacterTopologyCycle})  coincides, up to a scalar factor,  with the entire cyclic homology Chern character of $e$. 
\end{theorem}    
\begin{proof}
The relations
\begin{gather}
\tilde{b} (e(de)^{2n-1}) \;=\; 0 \\
d (e(de)^{2n-1})\; = \; e(de)^{2n}
\end{gather} 
show that   
\begin{equation}
e(de)^{2n} \;=\; \frac{1}{2}(e - \frac{1}{2}) (de)^{2n} - 
( \tilde{b} + d ) \frac {1}{2} (e(de)^{2n-1}). 
\end{equation}
This means the classes $e(de)^{2n}$,   $\frac{1}{2}(e - \frac{1}{2}) (de)^{2n}$
are co-homologous in the local and non-local complex 
$\sigma (\tilde{C}_{\ast}(\mathcal{A}))$. 
The non-local class $\frac{1}{2}(e - \frac{1}{2}) (de)^{2n}$ represents the Connes - Chern
character of the idempotent $e$ in the non-local entire cyclic homology, see Getzler, Senes \cite{Getzler_Szenes}  and \cite{Bondia_Figueroa_Varilly},  Pg.  447.
\end{proof}  
\begin{theorem}  
The Chern character defines a homomorphism
\begin{equation}
Ch \;:  T_{0}(\mathcal{A})  \longrightarrow  H^{per, loc}_{ev} (\mathcal{A})
\end{equation}
\end{theorem}   
\begin{proof}  
The Chern character is compatible with all equivalences which define 
$T_{0}(\mathcal{A})$:  stabilisation and conjugation. 
\end{proof}  

\section{Rational Pontrjagin classes of topological manifolds}
\subsection{Existance of direct connections on topological manifolds.}
We intend to construct a direct connection $A$ on $M$.  Recall that a \emph{direct connection}, see \cite{Teleman3}, on $M$ consists of a set  of \emph{isomorphisms} 
$A(x, y)$ where, in general
\begin{equation} 
A(x, y) = \text{isomorphism from a neighbourhood of $y$  to a neighbourhood of $x$}.
\end{equation} 
with the property
\begin{equation}
A(x, x) \;=\; Identity. 
\end{equation}
 The connection $A$ will be constructed by induction using a handlebody decomposition of $M$, see 
\cite{Kirby_Siebenmann_Essays}  7.1. Pg.   319.     
\par
The second condition on direct connections assures that the inductive construction of the connection $A$ may be performed without meeting homotopic obstructions in $\pi_{i} (TOP)$.
\subsection{The characteristic class}
We consider the chain
\begin{equation}
\Psi_{2n} (x_{0}, \dots x_{2n}) 
= N  \;  A(x_{i_{0}}, x_{i_{2n}}) \circ   d A(x_{i_{2n}}, x_{i_{2n-1}}) \dots   \circ  dA(x_{i_{0}}, x_{i_{2n}}),
\end{equation}
where $N$  is the cyclic  symmetrisation with respect to the variables $x$ and 
\begin{equation}
d A(y, x) \;= 1 \otimes A(y, x) - A(y, x) \otimes 1.
\end{equation}
Given that $A(x, x) = Identity$, the chain $\Psi_{2n}$ is a cycle in the $b^{'}$-complex. Therefore, 
\begin{enumerate}
\item
it is a $b^{'}$-cycle 
\item
it is $T$-invariant.
\end{enumerate}
$\Psi_{2n}$  is a cycle in the cyclic homology of the algebra $\mathcal{A}$. 



\begin{thebibliography}{99.}        

\bibitem{Swan}  Swan R., Vector bundles and projective modules. Trans. Amer. Math. Soc. 105, 264 - 277, 1962

\bibitem{Whitehead} Whitehead G. W.: Generalised homology theories. Trans. Amer. Math. Soc. , 102, pp. 227 - 283, 1962.


\bibitem{Teleman_1966} Teleman N.: A geometrical definition of some Andr�e Weil forms which can
be associated with an infinitesimal connection, (Roumanian). St. Cerc. Math. Tom. 18, No.
5, pp. 753-762, Bucarest, 1966.

\bibitem{Teleman_K}  Teleman K.: Sur le charact`ere de Chern d�un fibr�e complexe diff�erentiable,
Rev. Roumaine Math. Pures Applic. 12, pp. 725-731, 1967.


\bibitem{Milnor_Introduction} Milnor J.: Introduction to Algebraic $K$-Theory.  Annals of Math. Studies Vol. 72,  1971


\bibitem{Bass}  Bass H.: Introduction to some Methods of Algebraic $K$-theory. AMS, CBMS , 1974

\bibitem{Milnor} Milnor J.: Characteristic Classes, Annals of Mathematics Studies Nr. 76, Princeton, 1974    


\bibitem{Kirby_Siebenmann_Essays} Kirby R., Siebenmann L.: Foundational Essays on Topological Manifolds, Smoothings and Triangulations.  Annals of Mathematical Studies 88, Princeton Univ. Press, 1977
                                                                                          
\bibitem{Connes_IHES} Connes A.: Noncommutative differential Geometry,
Publ. Math. IHES 62, pp.257 - 360,  1985

\bibitem{Karoubi_1987} Karoubi M., Homologie cyclique et  K-Th$\acute{e}$orie. Ast$\acute{e}$risque No. 149, 147 pp., 1987

\bibitem{Connes1a} Connes A.: Entire cyclic cohomology of Banach algebras and characters of 
$\theta$-summable Fredholm modules, K-Theory 1, 519-548,1988

\bibitem{Blackadar} Blackadar B.: $K$-Theory for Operator Algebras, Second Ed, Cambridge University Press, 1998

\bibitem{Getzler_Szenes}  Getzler E.,  Szenes A.: On the Chern character of a theta-summable Fredholm module. J. Func. Anal. 84, 343-357,  1989.


\bibitem{Connes_Gromov_Moscovici} Connes A., Gromov M., Moscovici H.: Conjectures de Novikov et fibr$\acute{e}$s presque plats, C. R. Acad. Sci. Paris S$\acute{e}$r. A-B 310, 273-277, 1990


\bibitem{Connes_Moscovici} Connes A., Moscovici H.: Cyclic Cohomology, The Novikov Conjecture and Hyperbolic Groups,
Topology Vol. 29, pp. 345-388, 1990.

\bibitem{Loday} Loday J.-L.: Cyclic Homology, Grundlehren in mathematischen Wissenschaften 301, Springer Verlag,  Berlin Heidelberg, 1992.

\bibitem{Connes_Book} Connes A.: Noncommutative Geometry, Academic Press, 1994

\bibitem{Rosenberg} Rosenberg J.: Algebraic $K$-Theory and its Applications. Graduate Texts Nr. 147, Springer, Berlin,  1994.

\bibitem{Teleman_CR} Teleman N.:  Microlocalisation de l'Homologie de Hochschild, Compt. Rend. Acad. Scie. Paris,
Vol. 326, 1261-1264, 1998.

\bibitem{Bondia_Figueroa_Varilly}  Bondia Jose' M. Garcia - Bondia,   Figueroa Hector,  Varilly  Joseph C. Elements of Noncommutative Geometry,  Birkhauser Advanced Texts, 2000


 \bibitem{Rordam_Lansen_Lautsen} Rordam M., Lansen F., Lautsen N.: An Introduction to $K$-Theory for $C^{\ast}$-Algebras.
 London Math. Soc. Student Texts Nr. 49, Cambridge, 2000


\bibitem{Cuntz} Cuntz J.: Cyclic Theory, Bivariant $K$-theory and the bivariant Chern-Connes character, Operator Algebras in Noncommutative Geometry II, Encyclopedia of Mathematical Sciences, Vol. 121, Pg. 1 - 71, Springer Verlag, 2004 

\bibitem{Teleman3}Teleman N.: Direct Connections and Chern Character.
Proceedings of the International Conference in Honour of Jean-Paul Brasselet,
Luminy, World Scientific,  May 2005.

\bibitem{Teleman5} Teleman N.: Modified Hochschild and Periodic Cyclic Homology. Central European Journal of Mathematics. ``C*-Algebras and Elliptic Theory II'', 
Trends in Mathematics, 251-265, Birkhauser, 2008       

\bibitem{Teleman_arXiv_I} Teleman N.: $Local^{3}$ Index Theorem.  arXiv: 1109.6095v1, math.KT,   28 Sep. 2011.

\bibitem{Weibel} Weibel C.:  $K$-theory, 2012.

\bibitem{Teleman_arXiv_II} Teleman N.: {\em Local} Hochschild Homology of Hilbert-Schmidt Operators on Simplicial Spaces. 
             arXiv  hal-00707040, Version 1,  11 June 2012
           
\bibitem{Teleman_arXiv_III}  Teleman N.:     Local Algebraic $K$-Theory, 26 lug 2013 - arXiv.org.math.arXiv:1307.7014,  2013

\bibitem{Teleman_arXiv_IV}  Teleman N.:  The Local Index Theorem, HAL-00825083,  arXiv 1305.5329,  22 May 2013.
         
\bibitem{Teleman_arXiv_V}  Teleman N.:  $K_ {i}^{loc}(\mathbb {C}) $, $ i= 0, 1$
10 set 2013 - arXiv:1309.2421v1 [math.KT] 10 Sep 2013. 




\end{thebibliography}
\end{document}